\documentclass[a4paper,11pt]{article}

\usepackage[utf8]{inputenc}
\usepackage[T1]{fontenc}
\usepackage{textcomp}
\usepackage{lmodern}
\usepackage{amsmath,amsthm,esint}
\usepackage{graphicx}
\usepackage[textwidth=135mm, textheight=195mm]{geometry}
\usepackage[english]{babel}

\newtheorem{lemma}{Lemma}
\newtheorem{proposition}{Proposition}
\newtheorem{remark}{Remark}
\newtheorem{theorem}{Theorem}

\newcommand{\R}{\mathbf{R}}
\newcommand{\eps}{\epsilon}
\newcommand{\dx}{\,dx}
\newcommand{\dy}{\,dy}

\newcommand{\ds}{\,ds}
\newcommand{\dt}{\,dt}
\newcommand{\dsigma}{\,d\sigma}
\newcommand{\dtau}{\,d\tau}
\newcommand{\bigpars}[1]{\bigl(#1\bigr)}
\newcommand{\Bigpars}[1]{\Bigl(#1\Bigr)}
\newcommand{\biggpars}[1]{\biggl(#1\biggr)}
\newcommand{\abs}[1]{\lvert#1\rvert}
\newcommand{\biggabs}[1]{\biggl\lvert#1\biggr\rvert}
\newcommand{\norm}[1]{\lVert#1\rVert}
\newcommand{\set}[1]{\{#1\}}
\newcommand{\bigset}[1]{\bigl\{#1\bigr\}}

\newcommand{\lip}{\operatorname{lip}}
\newcommand{\dist}{\operatorname{dist}}
\newcommand{\linspan}{\operatorname{span}}
\DeclareMathOperator*{\argmin}{arg\,min}
\newcommand{\Set}[2]{\left\lbrace#1\ ;\ #2\right\rbrace}
\newcommand{\Scal}[2]{\left\langle#1\,,\,#2\right\rangle}
\newcommand{\smallParagraph}[1]{\paragraph{\textnormal{\textsl{#1}}}}

\newenvironment{dcases}%
{\everymath{\displaystyle\everymath{}}\begin{cases}}%
{\end{cases}}

\graphicspath{{img/}}

\title{A modified phase field approximation for mean curvature flow with
conservation of the volume}
\author{Elie Bretin \& Morgan Brassel\\[2ex]
LJK-IMAG, UMR 5523 CNRS,\\
51 rue des Math\'ematiques, B.P. 53, 38041 Grenoble cedex 9, France\\[1ex]
\texttt{elie.bretin@imag.fr} (corresponding author)\\
\texttt{morgan.brassel@imag.fr}
}
\date{April 2009}

\begin{document}

\maketitle

\begin{abstract}
This paper is concerned with the motion of a time dependent hypersurface
$\partial \Omega(t)$ in $\R^d$ that evolves with a normal velocity
\[
  V_n = \kappa - \fint_{\partial \Omega(t)} \kappa \dsigma,
\]
where $\kappa$ is the mean curvature of $\partial \Omega(t)$, and
$\fint_I$ stands for $\frac{1}{\abs{I}} \int_I$. Phase field
approximation of this motion leads to the nonlocal Allen--Cahn equation
\[
  \partial_t u = \Delta u - \frac{1}{\eps^2} W'(u)
    + \frac{1}{\eps^2} \fint_Q W'(u) \dx,
\]
where $Q$ is an open box of $\R^d$ containing $\partial \Omega(t)$ for all
$t$. We propose a modified version of this equation:
\[
  \partial_t u = \Delta u - \frac{1}{\eps^2} W'(u)
    + \frac{1}{\eps^2} \sqrt{2 W(u)}
    \biggpars{\int_Q \sqrt{2 W(u)} \dx}^{-1} \int_Q W'(u) \dx,
\]
and we show that it has better volume preserving properties than the
classical one, even in the presence of an additional forcing term $g$.
\end{abstract}

\section{Introduction and motivation}

In the last decades, a lot of work has been devoted to motions of
interfaces, and particularly to motion by mean curvature. Applications
concern image processing (denoising, segmentation), material sciences
(motion of grain boundaries in alloys, crystal growth), biology
(modelling of vesicles and blood cells).

In this paper, we are interested in phase field equations as an
approximation to motion by mean curvature with a forcing term and a
volume constraint.

For $t$ in $[0,T]$, let $\Omega(t)$ denote the evolution by mean curvature
with a forcing term of a smooth bounded domain $\Omega_0$ in $\R^d$.
More precisely, the normal velocity $V_n$, with normal $n$ pointing
towards the exterior of $\Omega(t)$, is given at a point $x$ of $\partial
\Omega(t)$ by
\begin{equation}
\label{eq:mc_motion_forcing}
  V_n = \kappa + g,
\end{equation}
where $\kappa$ denotes the mean curvature at $x$, with the convention
that $\kappa$ is negative if the set is convex, and where $g = g(x,t)$
is a given smooth forcing term. In this work, we only consider smooth
motions, which are well-defined if $T$ is sufficiently small
\cite{Ambrosio2000}. Singularities may develop in finite time, however,
and one may need to consider evolutions in the sense of viscosity
solutions \cite{Barles1994, Evans1992}.

The evolution of $\Omega(t)$ is closely related to the minimization of the
following energy:
\[
  J(\Omega) = \int_{\partial \Omega} 1 \dsigma - \int_\Omega g \dx.
\]
Indeed, one can view \eqref{eq:mc_motion_forcing} as a first order
optimality condition for this energy. The functional $J$ can be
approximated by a Ginzburg--Landau energy \cite{Modica1977,
Modica1977a}:
\[
  J_\eps(u)
  = \int_{\R^d} \biggpars{\frac{\eps}{2} \abs{\nabla u}^2
    + \frac{1}{\eps} W(u)} \dx - c_W \int_{\R^d} g u \dx,
\]
where $\eps$ is a small parameter, $W$ a double well potential with
wells at $0$ and $1$, for example $W(s) = \frac{1}{2} s^2
(1 - s)^2$, and where
\[
  c_W = \int_0^1 \sqrt{2 W(s)} \ds.
\]
Modica and Mortola \cite{Modica1977, Modica1977a} have shown the
$\Gamma$-convergence of $J_\eps$ to $c_W J$ in $L^1(\R^d)$ in the
absence of forcing terms (see also \cite{Bellettini1997}). The extension
of these results to motions with bounded forcing terms is
straightforward. The corresponding Allen--Cahn equation
\cite{Allen1979}, obtained as the gradient flow of $J_\eps$, reads
\begin{equation}
\label{eq:ac_forcing}
  \partial_t u
  = \Delta u - \frac{1}{\eps^2} W'(u) + \frac{1}{\eps} c_W g.
\end{equation}
This equation is usually solved in a fixed box $Q$ of $\R^d$, which
contains the motion $\Omega(t)$ for all $t$ in $[0,T]$. Existence, uniqueness
and a comparison principle have been established for this equation (see
for example chapters 14 and 15 in \cite{Ambrosio2000}). To this
equation, one usually associates the profile
\[
  q = \argmin \Set{\int_\R \biggpars{\frac{1}{2} {\gamma'}^2
    + W(\gamma)} \ds}
  {\gamma \in V},
\]
where $V$ is the space of functions in $H^1_{loc}(\R)$ that satisfies $\gamma(-\infty) = 1$,
$\gamma(+\infty) = 0$, $\gamma(0) = \frac{1}{2}$.
For $t$ in $[0,T]$, the motion $\Omega(t)$ can be approximated by that
of
\[
  \Omega_\eps(t) = \Set{x \in \R^d}{u_\eps(x,t) \ge \frac{1}{2}},
\]
where $u_\eps$ solves \eqref{eq:ac_forcing} with the initial condition
\[
  u_\eps(x,0) = q\biggpars{\frac{d(x, \Omega_0)}{\eps}}.
\]
Here $d(x, \Omega)$ denotes the signed distance of a point $x$ to the
set $\Omega$. The convergence of $\partial \Omega_{\eps}(t)$ to
$\partial \Omega(t)$ has been proved for smooth motions \cite{Chen1992,
Bellettini1995} and in the general case without fattening
\cite{Barles1994, Evans1992}. The rate of convergence has been proven to
be $O(\eps^2 \abs{\log \eps}^2)$. Actually, a formal asymptotic
expansion shows that $u_\eps$ behaves like
\begin{equation}
\label{eq:dev_asymp_u}
  u_\eps(x,t)
  = q\biggpars{\frac{d\bigpars{x, \Omega_\eps(t)}}{\eps}}
    + \eps g(x,t) \eta\biggpars{\frac{d\bigpars{x, \Omega_\eps(t)}}{\eps}}
    + O(\eps^2),
\end{equation}
where $\eta$ is defined as the solution in $H^2_{loc}(\R)$, with
polynomial growth, of
\begin{equation}
\label{eq:def_eta}
\begin{cases}
  \eta'' - W''(q) \eta = -c_W + q', \\
  \eta(0) = 0.
\end{cases}
\end{equation}
When $g = 1$, the modified profile $s \mapsto q_\eps(s) = q(s) + \eps
\eta(s)$ (see figure \ref{fig:profiles}) can be evaluated at $s = \pm
\infty$, where it takes the respective values
\[
  \eps \frac{c_W}{W''(1)}
  \quad \text{and} \quad
  1 + \eps \frac{c_W}{W''(0)}.
\]
These values correspond to the positions of the wells of a modified
double well potential $W_{\eps, g}$, defined by $W'_{\eps, g} = W' -
\eps c_W g$ and $W_{\eps, g}(0) = 0$.

\begin{figure}[htbp]
\centering
\includegraphics[width=.45\linewidth]{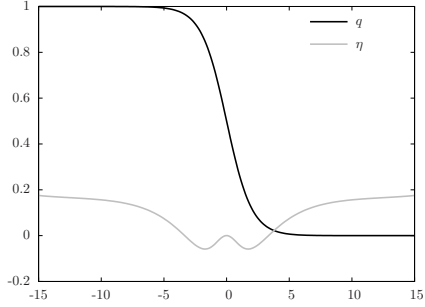}
\caption{Profile of the function $s \mapsto q(s) + \eps \eta(s)$ when
$W(s) = \frac{1}{2} s^2 (1-s)^2$.}
\label{fig:profiles}
\end{figure}

Our main interest is the numerical simulation of interfaces $\partial
\Omega(t)$ evolving from $\partial \Omega_0$ with normal velocity given by
\begin{equation}
\label{eq:mc_conserved_forcing}
  V_n = \kappa + g - \fint_{\partial \Omega(t)} (\kappa + g) \dsigma.
\end{equation}
In this case, it is easy to see that the volume of $\Omega(t)$,
\[
  \abs{\Omega(t)} = \int_{\Omega(t)} 1 \dx,
\]
remains constant in time. For instance, using the results in
\cite{Simon1980}, one may check that the shape derivative of the volume
is zero. The usual strategy to approximate
\eqref{eq:mc_conserved_forcing} is based on the remark that the mass
\[
  \int_{\R^d} u_\eps \dx
\]
is a good approximation of the volume $\abs{\Omega(t)}$. One can then add
to the Allen--Cahn equation an extra forcing term $\lambda(t)$,
independent of $x$, in order to impose the conservation of mass. This
leads to the following equation:
\[
  \partial_t u
  = \Delta u - \frac{1}{\eps^2} \Bigpars{W'(u) - \eps c_W g}
  + \frac{1}{\eps} c_W \lambda.
\]
The forcing term $\lambda$ can be viewed as a Lagrange multiplier
associated to the volume constraint. It can be determined by integrating
the equation over $Q$, which gives
\[
  \lambda = \frac{\eps}{c_W}
  \fint_Q \frac{1}{\eps^2} \Bigpars{W'(u) - \eps c_W g} \dx.
\]
In the case where $g = 0$, the previous equation reduces to
\begin{equation}
\label{eq:ac_conserved}
  \partial_t u = \Delta u - \frac{1}{\eps^2} W'(u)
    + \fint_Q \frac{1}{\eps^2} W'(u) \dx,
\end{equation}
which is the classical Allen--Cahn conserved equation (see
\cite{Rubinstein1992} and \cite{Bronsard1997}). Formally, one can think
of this equation as an approximation to motion by mean curvature with a
modified forcing term $g_\eps(t)$, independent of $x$ and given by
\[
  g_\eps = \frac{\eps}{c_W}
    \fint_Q \frac{1}{\eps^2} W'(u) \dx.
\]
In view of expansion \eqref{eq:dev_asymp_u}, one expects solutions of
\eqref{eq:ac_conserved} to behave like
\[
  u_\eps(x,t)
  = q\biggpars{\frac{d\bigpars{x, \Omega_\eps(t)}}{\eps}}
    + \eps g_\eps(t) \eta \biggpars{\frac{d\bigpars{x, \Omega_\eps(t)}}{\eps}}
    + O(\eps^2).
\]
By integration over $Q$, one sees that (see proposition
\ref{prop:volume} further)
\[
  \int_Q u_\eps \dx
  = \abs{\Omega_\eps(t)}
  + \eps g_\eps
    \int_Q \eta \biggpars{\frac{d\bigpars{x, \Omega_\eps(t)}}{\eps}} \dx
  + O(\eps^2).
\]
As the mass of $u_\eps$ is conserved, as $g_\eps = O(1)$, and as
\[
  \int_Q \eta \biggpars{\frac{d\bigpars{x, \Omega_\eps(t)}}{\eps}} \dx = O(1),
\]
given the values of $\eta$ at $\pm \infty$, one expects that
\[
  \abs{\Omega_\eps(t)} = \abs{\Omega_0} + O(\eps)
\]
only. This is not satisfactory for many applications, where loss of
volume during numerical computations strongly affects the dynamics.

The aim of this work is to propose another phase field model that has
better volume conservation properties than the conserved Allen--Cahn
equation. The paper is organised as follow:

In section \ref{sec:new_model}, we introduce the following phase field
approximation for mean curvature flow with a forcing term:
\begin{equation}
\label{eq:ac_forcing_new}
  \partial_t u = \Delta u - \frac{1}{\eps^2}
    \Bigpars{W'(u) - \eps \sqrt{2 W(u)} g}.
\end{equation}
It can be seen as the gradient flow of
\[
  \tilde{J}_\eps(u)
  = \int_{\R^d} \biggpars{\frac{\eps}{2} \abs{\nabla u}^2
    + \frac{1}{\eps} W(u)} \dx - \int_{\R^d} G(u) g \dx,
\]
with
\[
  G(s) = \int_0^s \sqrt{2 W(t)} \dt.
\]
We first explain via a formal asymptotic analysis why solutions of
\eqref{eq:ac_forcing_new} are expected to take the form
\begin{equation}
\label{eq:asymp_u_eps}
  u_\eps(x,t) = q\biggpars{\frac{d\bigpars{x, \Omega_\eps(t)}}{\eps}} + O(\eps^2).
\end{equation}
Then, following an argument due to \cite{Bellettini1995}, we rigorously
prove the convergence of this phase field equation to the motion
\eqref{eq:mc_motion_forcing}.

In section \ref{sec:conserv_vol}, we consider the evolution $\Omega(t)$
of a smooth bounded domain $\Omega_0$ according to
\[
  V_n = \kappa + g - \fint_{\partial \Omega} (\kappa + g) \dsigma.
\]
Let $u_\eps$ be a solution of
\begin{multline}
\label{eq:ac_conserved_forcing_new}
  \partial_t u = \Delta u - \frac{1}{\eps^2} \Bigpars{W'(u) - \eps \sqrt{2 W(u)} g} \\
    + \frac{1}{\eps^2} \frac{\sqrt{2 W(u)}}{\int_{\R^d} \sqrt{2 W(u)} \dx}
    \int_{\R^d} \Bigpars{W'(u) - \eps \sqrt{2 W(u)} g} \dx.
\end{multline}
We show that if $u_\eps$ behaves like in expansion \eqref{eq:asymp_u_eps},
then for all $t$ in $[0,T]$,
\begin{align*}
  \abs{\Omega_0}
  &= \int_{\R^d} u_\eps(x,0) \dx + O(\eps^2) \\
  &= \int_{\R^d} u_\eps(x,t) \dx + O(\eps^2) \\
  &= \abs{\Omega_\eps(t)} + O(\eps^2),
\end{align*}
while the solutions of \eqref{eq:ac_forcing} may only conserve volume up
to order $\eps$.

In section \ref{sec:numeric}, we present numerical evidence for the
above claims, which show that the modified phase field model
\eqref{eq:ac_conserved_forcing_new} has indeed better volume
preservation properties.

\section{A modified reaction--diffusion equation for mean curvature flow
with a forcing term}
\label{sec:new_model}

Let $\partial \Omega(t)$ denote an evolving hypersurface of codimension
$1$ in $\R^d$, with velocity law $V_n = \kappa + g$. This motion can
be interpreted as the energy gradient of
\[
  J(\Omega) = \int_{\partial \Omega} 1 \ds - \int_{\Omega} g \dx.
\]
Let $W$ be a bounded double well potential. In this whole section, we
will for convenience use a potential with wells at $-1$ and $1$, for
example $W(s) = \min \set{\frac{1}{2} (1 - s^2)^2, M}$, where $M$ is a
given positive constant. Our strategy is to introduce a modified
Ginzburg--Landau energy $\tilde{J}_\eps$ defined on $L^1(\R^d)$ by
\[
  \tilde{J}_\eps(u) =
  \begin{dcases}
    \int_{\R^d} \biggpars{\frac{\eps}{2} \abs{\nabla u}^2
    + \frac{1}{\eps} W(u)} \dx
    - \int_{\R^d} G(u) g \dx & \text{if $u \in H^1(\R^d)$},
    \\[2ex]
    +\infty & \text{otherwise},
  \end{dcases}
\]
with
\[
  G(s) = \int_0^s \sqrt{2 W(t)} \dt.
\]
The function $G$ is Lipschizt continuous. For $g$ in $L^\infty(\R^d)$,
the term $u \mapsto \int_{\R^d} G(u) g \dx$ acts as a continuous
perturbation in the $L^1(\R^d)$ topology of the classical
Modica--Mortola energy. The stability of $\Gamma$-convergence
with respect to continuous perturbations allows us to extend
the Modica--Mortola result to the case at hand, and show that
$\tilde{J}_\eps$ $\Gamma$-converges to $c_W J$. The gradient flow of
$\tilde{J}_\eps$ should then provide a mean to approximate the motion
of $\partial \Omega(t)$ via the resolution of the reaction--diffusion
equation \eqref{eq:ac_forcing_new}.

\begin{remark}
In the simplest case where $g = 1$, equations \eqref{eq:ac_forcing}
and \eqref{eq:ac_forcing_new} can be expressed as Allen--Cahn
equations with particular double well potentials respectively equal
to $W_{1, \eps}(s) = W(s) + \eps c_W s$, and $W_{2, \eps}(s) = W(s)
+ \eps G(s)$.  These two potentials are related through the position
and height of their wells, which are asymptotically equal as $\eps
\to 0$. This explains why we expect that \eqref{eq:ac_forcing} and
\eqref{eq:ac_forcing_new} converge to the same motion.
\end{remark}

\subsection{Formal asymptotics for the modified Allen--Cahn
equation}
\label{sec:formal_analysis}

We denote by $u_\eps$ the solution of equation
\eqref{eq:ac_forcing_new}:
\[
  \partial_t u = \Delta u - \frac{1}{\eps^2} W'(u)
  + \frac{1}{\eps} \sqrt{2 W(u)} g,
\]
with initial condition
\[
  u(x,0) = q\biggpars{\frac{d(x, \Omega_0)}{\eps}}.
\]
Our aim is to propose an asymptotic analysis of $u_\eps$ in the
simplest two-dimensional radial case. Using polar coordinates $(r,
\theta)$, we consider a forcing term $g$ which does not depend on
$\theta$: $g = g(r,t)$. The initial set $\Omega_0$ is taken as
a disk of radius $1$:
\[
  \Omega_0 = \Set{(r, \theta) \in [0, +\infty) \times [0,2\pi)}{r \le 1}.
\]
Let $\Omega(t)$ be the mean curvature flow evolving from $\Omega_0$
according to the law $V_n = \kappa + g$. It is well known that in
this case, $\Omega(t)$ remains a circle for all $t$ (recall that the
forcing term $g$ is supposed to be radial). We will denote by $R(t)$
the radius of $\Omega(t)$, solution of the following ODE:
\[
  R' + \frac{1}{R} = g(R,t),
\]
with initial condition $R(0) = 1$. In this simple case, the solution
$u_\eps$ is also radial and depends only on $r$. It satisfies
\[
  \partial_t u_\eps
  - \frac{1}{r} \partial_r (r \partial_r u_\eps)
  + \frac{1}{\eps^2} W'(u_\eps)
  - \frac{1}{\eps} \sqrt{2 W(u_\eps)} g = 0.
\]
As $u_\eps$ is radial, every of its level sets is circular,
and we denote by $R_\eps(t)$ the radius of $\set{u_\eps(r,t) =
\frac{1}{2}}$. We thus have $R_\eps(0) = 1$ and $u_\eps(R_\eps, t)
= \frac{1}{2}$. We introduce the classical stretched variable $y =
\frac{r - R_\eps}{\eps}$ (see \cite{Bellettini1995}), and we define
$U_\eps$ by
\[
  U_\eps(y,t) = u_\eps(R_\eps + \eps y, t).
\]
This new function $U_\eps$ satisfies
\begin{equation}
\label{eq:edp_U}
  \partial_t U_\eps
  - \frac{1}{\eps} R'_\eps \partial_y U_\eps
  - \frac{1}{\eps r} \partial_y U_\eps
  - \frac{1}{\eps^2} \partial_{yy} U_\eps
  + \frac{1}{\eps^2} W'(U_\eps)
  - \frac{1}{\eps} \sqrt{2 W(U_\eps)} g = 0.
\end{equation}
We now consider asymptotic developments of $U_\eps$ and $R_\eps$
as follow:
\[
  U_{\eps}(y,t) = \sum_{i=0}^{+\infty} \eps^i U_{i}(y,t), \quad
  R_{\eps}(t) = \sum_{i=0}^{+\infty} \eps^i R_{i}(t),
\]
with $U_0(0,t) = \frac{1}{2}$, $R_0(0) = R(0)$, and $U_i(0,t) = 0$,
$R_i(0) = 0$ for all $i \ge 1$. We have
\begin{gather*}
  \frac{1}{r}
    = \biggpars{\eps y + \sum_{i=0}^{+\infty} \eps^i R_i}^{-1}
    = \frac{1}{R_0} - \eps \frac{y + R_1}{{R_0}^2} + O(\eps^2),
  \\
  W'(U_\eps) = W'(U_0) + \eps W''(U_0) U_1
    + \eps^2 \Bigpars{W'''(U_0) U_1 + W''(U_0) U_2} + O(\eps^3),
  \\
  \sqrt{2 W(U_\eps)} = \sqrt{2W(U_0)}
    + \eps \frac{W'(U_0)}{\sqrt{2 W(U_0)}} U_1 + O(\eps^2),
  \\
  g(r, t)
    = g\biggpars{\eps y + \sum_{i=0}^{+\infty} \eps^i R_i, t}
    = g(R_0, t) + \eps \partial_r g(R_0, t) (y + R_1) + O(\eps^2).
\end{gather*}
Using these equalities, \eqref{eq:edp_U} rewrites
\begin{equation}
\label{eq:edp_U_power_of_eps}
\begin{aligned}
  0 &=
  \frac{1}{\eps^2} \Bigpars{\partial_{yy} U_0 - W'(U_0)} \\
  &\quad + \frac{1}{\eps} \biggpars{\partial_{yy} U_1 - W''(U_0) U_1
    + \partial_y U_0 \biggpars{R_0' + \frac{1}{R_0}}
    + \sqrt{2 W(U_0)} g(R_0, t) } \\
  &\quad - \partial_t U_0 + R'_1 \partial_y U_0 + R'_0 \partial_y U_1
    + \partial_{yy} U_2 + \frac{1}{R_0} \partial_y U_1
    - \frac{y + R_1}{{R_0}^2} \partial_y U_0 \\
  &\quad + g(R_0, t) \frac{W'(U_0)}{\sqrt{2 W(U_0)}} U_1
    + \partial_r g(R_0, t) (y + R_1) \sqrt{2 W(U_0)} \\
  &\quad - W'''(U_0) U_1 - W''(U_0) U_2 \\
  &\quad + O(\eps).
\end{aligned}
\end{equation}
Following powers of $\eps$, we will now identify each term to zero.

\smallParagraph{Terms in $\eps^{-2}$.}
The first term $U_0$ satisfies $\partial_{yy} U_0 = W'(U_0)$ with
initial condition $U_0(0, t) = \frac{1}{2}$ for all $t$ in $[0,T]$. It
can thus be identified to the profile $q$:
\[
  \forall y \in \R, \quad \forall t \in [0,T], \quad U_0(y, t) = q(y).
\]

\smallParagraph{Terms in $\eps^{-1}$.}
Knowing by definition of the profile that $q' = -\sqrt{2 W(q)}$,
it follows from $U_0(.,t) = q(.)$ that $\partial_y U_0 = - \sqrt{2
W(U_0)}$. Equation \eqref{eq:edp_U_power_of_eps} then gives
\[
  \partial_{yy} U_1 - W''(U_0) U_1
  = -\partial_y U_0 \biggpars{R_0' + \frac{1}{R_0} - g(R_0,t) }.
\]
Multiplying this equality by $\partial_y U_0$ and integrating over
$\R$, we get
\begin{align*}
  \biggpars{R_0' + \frac{1}{R_0} - g(R_0, t)}
    \int_\R (\partial_y U_0)^2 \dy
  &= -\int_\R \Bigpars{\partial_{yy} U_1 - W''(U_0) U_1}
    \partial_y U_0 \dy \\
  &= -\int_\R \partial_y \Bigpars{\partial_{yy} U_0 - W'(U_0)}
    U_1 \dy \\
  &= 0.
\end{align*}
As $\int_\R (\partial_y U_0)^2 \dy$ is strictly positive, we get the
following equation on $R_0$:
\[
  R_0' + \frac{1}{R_0} = g(R_0, t),
\]
with initial condition $R_0(0) = R(0)$. Hence $R_0$ can be identified
to $R$ since they both satisfy the same ODE with the same initial
datum. It follows from \eqref{eq:edp_U_power_of_eps} that $U_1$
is solution of
\[
  \partial_{yy} U_1 - W''(U_0) U_1 = 0.
\]
We then know (see section 3 in \cite{Bellettini1995}) that
there exists $\alpha(t) \in \R$ such that $U_1(., t) = \alpha(t)
q'(.)$. Indeed, the kernel of the operator $A: H^1(\R) \to H^{-1}(\R)$
defined by $A \zeta = \zeta'' - W''(q) \zeta$ can be identified to
$\linspan (q')$. Using $U_1(0,t) = 0$, we conclude that $\alpha(t)
= 0$ for all $t$, so that
\[
  \forall y \in \R, \quad \forall t \in [0,T], \quad U_1(y, t) = 0.
\]

\smallParagraph{Terms in $\eps^0$.}
Using $U_1 = 0$ and $\partial_t U_0 = 0$, we get from
\eqref{eq:edp_U_power_of_eps} that
\[
  \partial_{yy} U_2 - W''(U_0) U_2
  = \partial_y U_0 \biggpars{\frac{y + R_1}{{R_0}^2}
  - R'_1 + \partial_r g(R_0,t) (y + R_1)}.
\]
Multiplying by $\partial_y U_0$ and integrating over $\R$, we have
\begin{multline*}
  \biggpars{R'_1 - \frac{R_1}{{R_0}^2} - \partial_r g(R_0, t) R_1}
    \int_\R (\partial_y U_0)^2 \dy \\
  = -\int_\R \Bigpars{\partial_{yy} U_2 - W''(U_0) U_2 }
    \partial_y U_0 \dy
  + \biggpars{\frac{1}{{R_0}^2} + \partial_r g(R_0, t)}
    \int_\R y (\partial_y U_0)^2 \dy.
\end{multline*}
The first term in the right member vanishes as previously for
$U_1$. The second term also vanishes since $\partial_y U_0(., t)
= q'(.)$ is even with our choice of $W$. We deduce that $R_1$ is
solution of
\[
  R'_1 = \biggpars{\frac{1}{{R_0}^2} + \partial_r g(R_0, t)} R_1,
\]
with initial condition $R_1(0) = 0$. Hence $R_1(t) = 0$ for all $t$
in $[0,T]$. Finally, $U_2$ is obtained as the solution of
\[
  \partial_{yy} U_2 - W''(U_0) U_2
  = y \partial_y U_0 \biggpars{\frac{1}{{R_0}^2} + \partial_r g(R_0,t)}.
\]
Introducing the solution $\xi$ in $H^1(\R)$ of
\[
  \xi''(y) - W''\bigpars{q(y)} \xi(y) = y q'(y)
\]
with $\xi(0) = 0$, we can express $U_2$:
\[
  \forall y \in \R, \quad \forall t \in [0,T], \quad
  U_2(y,t) = \xi(y) \biggpars{\frac{1}{{R_0}^2} + \partial_r g(R_0,t)}.
\]

We finally conclude from this formal asymptotic analysis that $u_\eps$,
solution of \eqref{eq:ac_forcing_new}, is expected of the form
\begin{equation}
\label{eq:dev_u_complete}
  u_\eps(x,t)
  = q \biggpars{\frac{d\bigpars{x, \Omega_\eps(t)}}{\eps}}
  + \eps^2 \biggpars{\frac{1}{{R_0}^2} + \partial_r g(R_0,t)}
    \xi \biggpars{\frac{d\bigpars{x, \Omega_\eps(t)}}{\eps}}
  + O(\eps^3),
\end{equation}
where $\Omega_\eps(t)$ converge to $\Omega(t)$ in $O(\eps^2)$.

\subsection{Proof of convergence for the modified phase field model}

In this section, we closely follow the work of \cite{Bellettini1995}
to prove the following theorem:

\begin{theorem}
\label{th:conv_forcing}
Let $\Omega(t)$ be a regular mean curvature flow with a forcing term
$g$ that satisfies
\begin{equation}
\label{eq:hyp_g}
  g(.,t) \in W^{3, \infty}(\R^d),
  \quad \partial_t g \in W^{1, \infty}\bigpars{\R^d \times (0,T)}.
\end{equation}
Given $\eps > 0$, let $u_\eps$ be solution of
\eqref{eq:ac_forcing_new}:
\[
  \partial_t u = \Delta u - \frac{1}{\eps^2} W'(u)
  + \frac{1}{\eps} \sqrt{2 W(u)} g,
\]
and let $\partial \Omega_\eps(t) = \Set{x \in \R^d}{u_\eps(x,t)
= \frac{1}{2}}$. Assume that the potential $W$ is given by $W(s) =
\frac{1}{2} (1 - s^2)^2$. Then there exist $\eps_0 > 0$ and a constant
$C$ depending only on $T$ such that for all $\eps$ in $(0,\eps_0]$,
the following estimate holds:
\begin{equation}
\label{eq:inclusion_th}
  \forall t \in [0,T], \quad
  \partial \Omega_\eps(t)
  \subseteq \Set{x \in \R^d}{\dist \bigpars{x, \partial \Omega(t)}
    \le C \eps^2 \abs{\log \eps}^2}.
\end{equation}
\end{theorem}

\paragraph{Notations and assumptions.}
Let $T>0$. For all $t$ in $[0,T]$, let $\Omega(t)$ be a mean curvature
flow with a forcing term $g$ that satisfies \eqref{eq:hyp_g}. In the
sequel, we will for convenience identify the signed distance to
$\Omega(t)$ to a function $d \colon \R^d \times [0,T] \to \R$ defined by
\[
  d(x,t) = d\bigpars{x, \Omega(t)} =
  \begin{cases}
    \dist \bigpars{x, \Omega(t)}
      & \text{if $x \in \R^d \setminus \Omega(t)$}, \\
    0
      & \text{if $x \in \partial \Omega(t)$}, \\
    -\dist \bigpars{x, \Omega(t)}
      & \text{if $x \in \Omega(t)$}.
  \end{cases}
\]
We assume that $\partial \Omega(t)$ is smooth enough so that $d$
satisfies
\begin{equation}
\label{eq:hyp_dist}
  d,\ \partial_t d,\ \partial_t \partial_{xx} d
  \in C^0(\bar{\Lambda}),
\end{equation}
where $\bar{\Lambda}$ is a tubular neighborhood of $\partial
\Omega(t)$. We assume that $\partial \Omega(t)$ is oriented by
the outward normal vector $n$ defined at a point $x$ of $\partial
\Omega(t)$ by $n(x,t) = \nabla d(x,t)$. We denote by $\kappa_1,
\dots, \kappa_{d-1}$ the principal curvatures of $\partial \Omega(t)$,
and we set
\[
  \kappa(x,t) = \sum_{i=1}^{d-1} \kappa_i(x,t),
  \quad
  h(x,t) = \sum_{i=1}^{d-1} \kappa^2_i(x,t).
\]
We choose $\kappa$ to be negative for convex balls. The evolution of
$\partial \Omega(t)$ is defined by $V_n(x,t) = \kappa(x,t) + g(x,t)$
for all $(x,t)$ in $\partial \Omega(t) \times [0,T]$, where $V_n$
denote the normal velocity.

Given $D > 0$, we define a tubular neighborhood $\Lambda(t)$ of
$\partial \Omega(t)$ by
\begin{equation}
\label{eq:def_Lambda}
  \Lambda(t) = \Set{x \in \R^d}{\abs{d(x,t)} \le D},
\end{equation}
and we set
\[
  \Lambda = \bigcup_{t \in [0,T]} \Lambda(t) \times \set{t}.
\]
If $D$ is sufficiently small, one can associate to any point $(x,t)$
of $\Lambda$ a unique projection $s(x,t)$ on $\partial \Omega(t)$
such that
\[
  \dist\bigpars{s(x,t),x} = \abs{d(x,t)}.
\]
For any scalar or vector function $f$ defined on $\partial \Omega(t)$,
we denote by $\bar{f}$ its extension on $\Lambda$, defined by
$\bar{f}(x,t) = f\bigpars{s(x,t),t}$. If $f$ is real-valued, then
we clearly have $\nabla d \cdot \nabla \bar{f} = 0$ on $\Lambda$. It
follows from \eqref{eq:hyp_dist} that
\begin{equation}
\label{eq:prop_h}
  \norm{\bar{h}}_{L^\infty(\Lambda)},
  \ \norm{\partial_t \bar{h}}_{L^\infty(\Lambda)},
  \ \norm{\nabla \bar{h}}_{L^\infty(\Lambda)},
  \ \norm{\Delta \bar{h}}_{L^\infty(\Lambda)} < +\infty.
\end{equation}
Moreover, geometric properties of the distance function $d$ imply
\begin{align*}
  \Delta d(x,t)
    &= \sum_{i=1}^{d-1}
      \frac{-\bar{\kappa}_i(x,t)}{1 - d(x,t) \bar{\kappa}_i(x,t)}
    = -\bar{\kappa}(x,t) - d(x,t) \bar{h}(x,t)
      + O\bigpars{d(x,t)^2}, \\
  \partial_t d(x,t)
    &= -\bar{V_n}(x,t)
    = -\bar{\kappa}(x,t) - \bar{g}(x,t).
\end{align*}
These estimates show that the motion of $\partial \Omega(t)$
can be described by an equation on $d$ inside the whole $\Lambda$
(see \cite{Ambrosio2000}):
\begin{equation}
\label{eq:distance}
  \forall (x,t) \in \Lambda, \quad
  \partial_t d(x,t) - \Delta d(x,t)
  = -\bar{g}(x,t) + d(x,t) \bar{h}(x,t) + O\bigpars{d(x,t)^2}.
\end{equation}

We denote by $q$ the profile function associated with the double well
potential $W$:
\[
  q = \argmin_{\zeta} \Set{\int_\R
  \biggpars{\frac{1}{2} {\zeta'}^2 + W(\zeta)} \ds}
  {\zeta \in H^1_{loc}(\R),
   \ \lim_{x \to \pm\infty} \zeta = \mp 1,
   \ \zeta(0) = 0}.
\]
The Euler equation for this problem writes $q'' = W'(q)$. More
precisely, as $W$ is smooth, $q$ is strictly decreasing and we have
$q' = -\sqrt{2 W(q)}$. If $W$ is defined by $W(s) = \frac{1}{2}
(1 - s^2)^2$, $q$ is given by $q(s) = -\tanh(s)$. In this case,
there exists a positive constant $C$ such that $\abs{q - 1} \le -C q'$.

Let $\xi$ in $H^2(\R)$ be solution of equation
\begin{equation}
\label{eq:def_xi}
  \xi''(s) - W''\bigpars{q(s)} \xi(s) = s q'(s)
\end{equation}
with initial condition $\xi(0) = 0$. Existence and uniqueness results
for this equation may be found in section 3 of \cite{Bellettini1995},
along with the following estimate:
\begin{equation}
\label{eq:estim_xi}
  \abs{\xi(s)},\ \abs{\xi'(s)} \le -C (1 + s^2) q'(s).
\end{equation}

\paragraph{Comparison lemma.}
Our proof of convergence relies on the following lemma:

\begin{lemma}
\label{lem:comparison}
Let $\eps > 0$, and let $u$ and $v$ in $L^2\bigpars{0,T; H^2(\R^d)}
\cap H^1\bigpars{0,T; L^2(\R^d)}$ be such that
\begin{equation}
\label{eq:hyp_lemma}
  \partial_t u - \Delta u + \frac{1}{\eps^2} W'(u)
    - \frac{1}{\eps} \sqrt{2 W(u)} g
  \ge \partial_t v - \Delta v + \frac{1}{\eps^2} W'(v)
    - \frac{1}{\eps} \sqrt{2 W(v)} g
\end{equation}
in $\R^d \times (0,T)$, and $u(x,0) \ge v(x,0)$ for $x$ in $\R^d$.
Then $u \ge v$ in $\R^d \times (0,T)$.
\end{lemma}

\begin{proof}
Let $e = \max(v - u, 0)$. Multiplying \eqref{eq:hyp_lemma} by $e$
and integrating over $\R^d$, we get
\begin{align*}
  \frac{d}{dt} {\norm{e(.,t)}}^2_{L^2(\R^d)}
  & \le \frac{2}{\eps^2} \Scal{W'(u) - W'(v)}{e}_{L^2(\R^d)} \\
  &\qquad  - \frac{2}{\eps}
    \Scal{\Bigpars{\sqrt{2 W(u)} - \sqrt{2 W(v)}} g}{e}_{L^2(\R^d)} \\
  & \le \frac{2}{\eps^2}
    \Scal{W'_{g, \eps}(u) - W'_{g, \eps}(v)}{e}_{L^2(\R^d)},
\end{align*}
where $W_{g, \eps}$ is defined by
\[
  W'_{g, \eps}(s)
  = W'(s) - \eps \sqrt{2 W(s)} g, \quad W(0) = 0.
\]
The idea is then to decompose $W'_{g, \eps}$ under the form
\[
  W'_{g, \eps} = W'_{g, \eps, L} + W'_{g, \eps, I},
\]
where $W'_{g, \eps, L}$ is Lipschitz continuous on $\R$, and $W'_{g,
\eps, I}$ is nondecreasing. More precisely, when $W(s) = \frac{1}{2}
(1 - s^2)^2$, we can use
\[
  W'_{g, \eps, L}(s) = W'_{g, \eps}(s) \chi_{[-1,1]}(s)
  \quad \text{and} \quad
  W'_{g, \eps, I}(s) = W'_{g, \eps}(s) \bigpars{1 - \chi_{[-1,1]}(s)},
\]
which satisfy the previous assumption if $\norm{g}_{L^\infty} \le
\frac{2}{\eps}$. Then, noticing that $e(x, 0) = 0$ by assumption,
we obtain
\begin{align*}
  \norm{e(.,t)}^2_{L^2(\R^d)}
    & \le \frac{2}{\eps^2} \int_0^t
      \biggabs{\Scal{W'_{g, \eps}(u) - W'_{g, \eps}(v)}
      {e(.,\tau)}_{L^2(\R^d)}} \dtau \\
    & \le \frac{2}{\eps^2} \sup_{x \in \R^d}
      \bigset{\lip(W'_{g, \eps, L})}
      \int_0^t \norm{e(.,\tau)}^2_{L^2(\R^d)} \dtau.
\end{align*}
Note that the $\sup$ is bounded just as
$\norm{g}_{L^\infty}$. Gronwall's lemma implies that for almost
every $t$ in $(0,T)$, $\norm{e(.,t)}_{L^2(\R^d)} = 0$, and $e = 0$
almost everywhere in $\R^d \times (0,T)$.
\end{proof}

\paragraph{Construction of a subsolution.}
Using our previous asymptotic expansion of $u_\eps$, we now build a
subsolution to problem \eqref{eq:ac_forcing_new}. Let $\delta \ge
3$ be a fixed integer. For all $\eps > 0$, we set $s_\eps = \delta
\abs{\log \eps}$. Since $q(s) = - \tanh(s)$, we have
\[
  q(s_\eps)
    = -1 + \frac{2 \eps^{2 \delta}}{1 + \eps^{2 \delta}}
    = -1 + O(\eps^{2 \delta}), \quad
  q'(s_\eps)
    = -\bigpars{1 - q(s_\eps)^2}
    = O(\eps^{2 \delta}),
\]
and it follows from \eqref{eq:estim_xi} that
\[
  \abs{\xi(s_\eps)} = O(\eps^{2 \delta} \abs{\log \eps}^2), \quad
  \abs{\xi'(s_\eps)} = O(\eps^{2 \delta} \abs{\log \eps}^2).
\]
We define two auxiliary functions $q_\eps$ and $\xi_\eps$ by
\[
  q_\eps(s) =
  \begin{cases}
    q(s) & \text{if } 0 \le s \le s_\eps, \\
    P_q(s) & \text{if } s_\eps \le s \le 2 s_\eps, \\
    -1 & \text{if } s > 2 s_\eps, \\
    - q_\eps(-s) & \text{if } s < 0,
  \end{cases}
\]
and
\[
  \xi_\eps(s) =
  \begin{cases}
    \xi(s) & \text{if } 0 \le s \le s_\eps, \\
    P_\xi(s) & \text{if } s_\eps \le s \le 2 s_\eps, \\
    0 & \text{if } s > 2 s_\eps, \\
    - \xi_\eps(-s) & \text{if } s < 0,
  \end{cases}
\]
where $P_q$ and $P_\xi$ are polynomials of degree $3$ defined in such
a way that $q_\eps$ and $\xi_\eps$ are in $C^1(\R)$. It follows that
\begin{align*}
  \norm{P_q + 1}_{L^\infty(I_\eps)}
    + s_\eps \norm{P'_q}_{L^\infty(I_\eps)}
    + s_\eps^2 \norm{P''_q}_{L^\infty(I_\eps)}
    &\le C \Bigpars{\abs{q(s_\eps) + 1} + s_\eps \abs{q'(s_\eps)}}, \\
  \norm{P_\xi}_{L^\infty(I_\eps)}
    + s_\eps \norm{P'_\xi}_{L^\infty(I_\eps)}
    + s_\eps^2 \norm{P''_\xi}_{L^\infty(I_\eps)}
    &\le C \Bigpars{\abs{q(s_\eps)} + s_\eps \abs{q'(s_\eps)}},
\end{align*}
with $I_\eps = [s_\eps, 2 s\eps]$.
Then we easily check that
\[
  \norm{q_\eps - q}_{L^\infty(\R)} = o(\eps^{2\delta - 1}), \quad
  \norm{\xi_\eps - \xi}_{L^\infty(\R)} = o(\eps^{2\delta - 1}),
\]
together with
\begin{equation}
\label{eq:edo_q_eps}
  q''_\eps - W'(q_\eps) = o(\eps^{2 \delta - 1}), \quad
  q'_\eps + \sqrt{2 W(q_\eps)} = o(\eps^{2 \delta - 1}),
\end{equation}
and
\[
  \xi''_\eps - W''(q_\eps) \xi_\eps - s q'_\eps
  = o(\eps^{2 \delta - 1}).
\]

For $\eps > 0$, we introduce the modified distance function $d_\eps^-$
defined by:
\[
  \forall (x,t) \in \R^d \times [0,T], \quad
  d_\eps^-(x,t) = d(x,t) + c_1(t) \eps^2 \abs{\log \eps}^2,
\]
where $c_1$ is a positive continuous function, independent of $\eps$,
that will be determined later. For $t$ in $[0,T]$, we introduce the sets
\[
  \Lambda_\eps^-(t)
  = \Set{x \in \R^d}{\abs{d^-_\eps(x,t)}
  < 2 \delta \eps \abs{\log \eps}}
\]
and
\[
  \Lambda^-_\eps
  = \bigcup_{t \in [0,T]} \Lambda^-_\eps(t) \times \set{t}.
\]
It is then possible to find $\eps_0 > 0$ depending only on $\delta$,
$c_1$, $D$, such that
\begin{equation}
\label{eq:incl_Lambda}
  \forall \eps \le \eps_0,
  \quad \forall t \in [0,T],
  \quad \Lambda_\eps^-(t) \subset \Lambda(t),
\end{equation}
where $\Lambda(t)$ is the tubular neighborhood defined in
\eqref{eq:def_Lambda}. In particular, we see that
\[
  \forall (x,t) \in \Lambda_\eps^-,
  \quad d(x,t) = O(\eps \abs{\log \eps}).
\]
Noticing that $\nabla d_\eps^- = \nabla d$ and $\nabla d_\eps^- \cdot
\nabla \bar{h} = 0$ in $\Lambda_\eps^-$, it follows from
\eqref{eq:distance} that
\begin{align}
\label{eq:modified_distance}
  \partial_t d_\eps^- - \Delta d_\eps^-
    &= \partial_t d - \Delta d + c_1' \eps^2 \abs{\log \eps}^2 \notag \\
    &= -\bar{g} + d_\eps^- \bar{h} + (c_1' -c_1 \bar{h}) \eps^2 \abs{\log \eps}^2
       + O(\eps^2 \abs{\log \eps}^2).
\end{align}
Setting $y = \frac{d_\eps^-}{\eps}$, we define $v_\eps^-$ on $\R^d
\times [0,T]$ by
\[
  v^-_\eps =
  \begin{cases}
    q_\eps(y) + \eps^2 (\bar{h} + \nabla d \cdot \nabla g) \xi_\eps(y)
      - c_2 \eps^3 \abs{\log \eps}^2
      & \text{in } \Lambda_\eps^-, \\
    -1 - c_2 \eps^3 \abs{\log \eps}^2
      & \text{in } \set{d_\eps^- \ge 2 \delta \eps \abs{\log \eps}}, \\
    +1 - c_2 \eps^3 \abs{\log \eps}^2
      & \text{in } \set{d_\eps^- \le -2 \delta \eps \abs{\log \eps}},
  \end{cases}
\]
where $c_2$ is a constant independent of $\eps$ that we will be
determined later. In view of \eqref{eq:hyp_dist}, we easily check that
$v_\eps^-$ belongs to $L^2 \bigpars{0,T; H_{loc}^1(\R^d)} \cap H^1
\bigpars{0,T; L_{loc}^2(\R^d)}$. Our goal is to show that $v_\eps^-$ is
a subsolution of \eqref{eq:ac_forcing_new}.

Let $u_\eps$ be solution of \eqref{eq:ac_forcing_new}. We will first
prove that
\begin{equation}
\label{eq:comparison_uv}
  \forall x \in \R^d, \quad v_\eps^-(x,0) \le u_\eps(x,0).
\end{equation}
To this end, we introduce $w_\eps$ defined by
\[
  w_\eps = q(y)
    + \eps^2 (\bar{h} + \nabla d \cdot \nabla g) \xi(y)
    - \frac{c_2}{2} \eps^3 \abs{\log \eps}^2,
\]
and we note that when $\eps$ is sufficiently small,
\begin{align*}
  v_\eps(x, 0)
  &\le w_\eps(x, 0) - \frac{c_2}{2} \eps^3 \abs{\log \eps}^2
    + o(\eps^{2 \delta - 1}) \\
  &\le w_\eps(x, 0),
\end{align*}
so that \eqref{eq:comparison_uv} follows from showing that
\begin{multline*}
  w_{\eps}(x,0) - u_{\eps}(x,0)
  = q\bigpars{y(x,0)} - q\biggpars{\frac{d(x,0)}{\eps}} \\
  + \eps^2 \Bigpars{\bar{h}(x,0)
    + \nabla d(x,0) \cdot \nabla g(x,0)} \xi\bigpars{y(x,0)}
    - \frac{c_2}{2} \eps^3 \abs{\log \eps}^2
\end{multline*}
is non-positive. We define for convenience
\begin{align*}
  I_1 &= q\bigpars{y(x,0)} - q\biggpars{\frac{d(x,0)}{\eps}}, \\
  I_2 &= \eps^2 \Bigpars{\bar{h}(x,0)
    + \nabla d(x,0) \cdot \nabla g(x,0)} \xi\bigpars{y(x,0)}.
\end{align*}
The following lemma is proved in section 6 of \cite{Bellettini1995}
(recall that $q'$ is negative).

\begin{lemma}
Let $z = \frac{d(x,0)}{\eps}$, and let $y = \frac{d_\eps^-(x,0)}{\eps} =
z + c_1(0) \eps \abs{\log \eps}^2$. Then for $\eps$ sufficiently small,
\[
  2 q'(y) \le q'(s) \le \frac{1}{2} q'(y)
\]
for all $s$ in $[z,y]$.
\end{lemma}

This lemma implies that
\begin{align*}
  I_1
  & = q\biggpars{\frac{d_\eps^-(x,0)}{\eps}}
    - q\biggpars{\frac{d_\eps^-(x,0) - c_1(0) \eps^2
    \abs{\log \eps}^2}{\eps}} \\
  & \le \frac{1}{2} q'\biggpars{\frac{d_\eps^-(x,0)}{\eps}}
    c_1(0) \eps \abs{\log \eps}^2,
\end{align*}
and using \eqref{eq:def_xi}, it follows that
\[
  I_2 \le -K c \eps^2 \bigpars{1 + y(x,0)^2}
    q'\bigpars{y(x,0)},
\]
where $K = \norm{h(.,0)}_{L^{\infty}(\Lambda(0))} + \norm{\nabla
g(.,0)}_{L^{\infty}(\R^d)}$. We then distinguish two cases. If $\abs{y}
> \abs{\log \eps}$, then $(1 + y^2) \abs{q'(y)} < O(\eps^2 \abs{\log
\eps}^2)$ and $I_2$ is controlled by the negative term $-\frac{c_2}{2}
\eps^3 \abs{\log \eps}^2$, so that
\begin{align*}
  w_\eps(x,0) - u_\eps(x,0)
  & = I_1 + I_2 - \frac{c_2}{2} \eps^3 \abs{\log \eps}^2 \\
  & \le I_2 - \frac{c_2}{2} \eps^3 \abs{\log \eps}^2 \\
  & \le O(\eps^4 \abs{\log \eps}^2)
    - \frac{c_2}{2} \eps^3 \abs{\log \eps}^2.
\end{align*}
If $\abs{y} < \abs{\log \eps}$, then $I_2$ is controlled by $I_1$, and
\begin{align*}
  w_\eps(x,0) - u_\eps(x,0)
  & = I_1 + I_2 - \frac{c_2}{2} \eps^3 \abs{\log \eps}^2 \\
  & \le q'\biggpars{\frac{d_\eps^-(x,0)}{\eps}}
    \biggpars{\frac{1}{2} c_1(0) \eps \abs{\log \eps}^2
    - O(\eps^2 \abs{\log \eps}^2)}.
\end{align*}
Thus, choosing $c_1(0)$ and $c_2$ sufficiently large, we get the desired
estimate \eqref{eq:comparison_uv}.

Let us now check that
\begin{equation}
\label{eq:v_subsolution}
  \partial_t v_\eps^- - \Delta v_\eps^-
  + \frac{1}{\eps^2} W'(v_\eps^-)
  - \frac{1}{\eps} \sqrt{2 W (v_\eps^-)} g \le 0
\end{equation}
in $\R^d \times (0,T)$.

\smallParagraph{Case 1: $(x,t) \in \Lambda_\eps^-$.}
In this case, \eqref{eq:prop_h} implies that
\begin{align*}
  \partial_t v_\eps^-
    & = \frac{1}{\eps} q_\eps'(y) \partial_t d_\eps^-
    + \eps^2 \Bigpars{\partial_t (\bar{h} + \nabla d \cdot \nabla g)}
    \xi_\eps(y) + \eps (\bar{h} + \nabla d \cdot \nabla g)
    \xi'_\eps(y) \partial_t d_\eps^-,
  \\
  \nabla v_\eps^-
    & = \frac{1}{\eps} q_\eps'(y) \nabla d_\eps^-
    + \eps^2 \Bigpars{\nabla (\bar{h} + \nabla d \cdot \nabla g)}
    \xi_\eps(y) + \eps (\bar{h} + \nabla d \cdot \nabla g)
    \xi'_\eps(y) \nabla d_\eps^-,
  \\
  \Delta v_\eps^-
    & = \frac{1}{\eps^2} q_\eps''(y)
    + \frac{1}{\eps} q_\eps'(y) \Delta d_\eps^-
    + (\bar{h} + \nabla d \cdot \nabla g) \xi''_\eps(y) + O(\eps).
\end{align*}
Using these equalities with \eqref{eq:modified_distance}, we get
\begin{align*}
  \partial_t v_\eps^- - \Delta v_\eps^-
  &= -\frac{1}{\eps^2} q_\eps''(y)
    + \frac{1}{\eps} q_\eps'(y) (\partial_t d_\eps^- - \Delta d_\eps^-)
    - (\bar{h} + \nabla d \cdot \nabla g) \xi''_\eps(y) + O(\eps)
  \\
  &= -\frac{1}{\eps^2} q_\eps''(y) + \frac{1}{\eps} q_\eps'(y)
    \Bigpars{-\bar{g} + \eps y \bar{h}
    - \eps^2 \abs{\log \eps}^2 (c_1 \bar{h} - c_1')} \\
    & \quad \qquad -(\bar{h} + \nabla d \cdot \nabla g) \xi''_\eps(y)
    + O(\eps \abs{\log \eps}^2) \\
  &= \frac{1}{\eps^2} q_\eps''(y)
    - \frac{1}{\eps} q_\eps'(y) \bar{g}
    + \Bigpars{q_\eps'(y) y \bar{h}
      - (\bar{h} + \nabla d \cdot \nabla g) \xi''_\eps(y)} \\
    & \quad \qquad - \eps \abs{\log \eps}^2
    \Bigpars{q_\eps'(y) (c_1 \bar{h} - c_1')}
    + O(\eps \abs{\log \eps}^2).
\end{align*}
Since $(x,t)$ is in $\Lambda_\eps^-$, we obtain the following estimates
on the terms of order $0$ in \eqref{eq:v_subsolution}:
\begin{gather*}
\begin{aligned}
  g(x,t)
  & = g\bigpars{s(x,t) + d \nabla d, t} \notag \\
  & = \bar{g} + d \nabla d \cdot \bar{\nabla g } + O(d^2) \notag \\
  & = \bar{g} + \eps y \nabla d^-_\eps \cdot \bar{\nabla g}
    - c_1 \eps^2 \abs{\log \eps}^2 \nabla d^-_\eps \cdot \bar{\nabla g}
    + O(\eps^2 \abs{\log \eps}^2),
\end{aligned}
\\
  \frac{1}{\eps^2} W'(v_\eps^-)
  = \frac{1}{\eps^2} W'(q_\eps) + W''(q_\eps)
    (\bar{h} + \nabla d \cdot \nabla g) \xi_\eps
    - W''(q_\eps) c_2 \eps \abs{\log \eps}^2 + O(\eps),
\\
\begin{aligned}
  \frac{1}{\eps} \sqrt{2 W (v_\eps)} g
  & = \frac{1}{\eps} \sqrt{2 W (q_\eps)} g + O(\eps) \\
  & = \frac{1}{\eps} \sqrt{2 W (q_\eps)}
    \Bigpars{\bar{g} + \eps y \nabla d_\eps \cdot \bar{\nabla g}
    - c_1 \eps^2 \abs{\log \eps}^2 \nabla d_\eps^- \cdot \bar{\nabla g}} \\
    &\qquad + O(\eps \abs{\log \eps}^2) \\
  & = -\frac{1}{\eps} q'_\eps \Bigpars{\bar{g}
    + \eps y \nabla d_\eps^- \cdot \bar{\nabla g}
    - c_1 \eps^2 \abs{\log \eps}^2 \nabla d_\eps^- \cdot \bar{\nabla g}}
    + O(\eps \abs{\log \eps}^2) \\
  & = -\frac{1}{\eps} q'_\eps \Bigpars{\bar{g}
    + \eps y \nabla d_\eps^- \cdot \nabla g
    - c_1 \eps^2 \abs{\log \eps}^2 \nabla d \cdot \bar{\nabla g}}
    + O(\eps \abs{\log \eps}^2).
\end{aligned}
\end{gather*}
Here, we used \eqref{eq:edo_q_eps} and the fact that for $(x,t)$ in
$\Lambda_\eps^-$,
\[
  y q'_\eps \nabla d \cdot (\nabla g - \bar{\nabla g})
  = O(\eps \abs{\log \eps}).
\]
Summing the previous equalities, we obtain
\[
  \partial_t v_\eps - \Delta v_\eps + \frac{1}{\eps^2} W'(v_\eps)
  - \frac{1}{\eps} \sqrt{2 W (v_\eps)} g
  = I_1 + I_2 + I_3 + I_4 + O(\eps \abs{\log \eps}^2),
\]
with
\begin{align*}
  I_1 &= -\frac{1}{\eps^2} \Bigpars{ q_\eps''(y) - W'\bigpars{q_\eps(y)} }
    = o(\eps^{2 \delta - 3}), \\
  I_2 &= \frac{1}{\eps} q'_\eps ( \bar{g} - \bar{g} )
    = 0, \\
  I_3 &= -(\bar{h} + \nabla d \cdot \nabla g)
    \Bigpars{ \xi''_\eps(y) - W''(q_\eps )\xi_\eps - y q'_\eps }
    = o(\eps^{2 \delta - 3}), \\
  I_4 &= \eps \abs{\log \eps}^2
    \Bigpars{ -q_{\eps}'(c_1\bar{h} - c_1' + c_1 \nabla d \cdot \bar{\nabla g} )
    - c_2 W''(q_{\eps}) }.
\end{align*}
We now determine the function $c_1$ and the constant $c_2$ so that $I_4$
is sufficiently negative to compensate the term of order $\eps \abs{\log
\eps}^2$. Letting $K = \norm{h}_{L^\infty(\partial \Omega)} +
\norm{\nabla g}_{L^\infty(\R^d \times (0,T))}$, we set
\[
  c_1(t) = c \exp\bigpars{(1+K)t},
\]
so that
\[
  -\eps \abs{\log \eps}^2 q_\eps'\Bigpars{c_1 \bar{h} - c_1'
    + \nabla d \cdot \bar{\nabla g} c_1}
  \le c_1 \eps \abs{\log \eps}^2 q_\eps'.
\]
We thus have
\[
  I_4 \le -\eps \abs{\log \eps}^2
    c_2 \biggpars{-\frac{c_1}{c_2} q_\eps' + W''(q_\eps)}.
\]
Noticing that $-c_3 q_\eps' + W''(q_\eps)$ is uniformly positive for
$c_3$ large enough, we can choose $c$ and $c_2$ such that
\[
  \partial_t v_\eps^- - \Delta v_\eps^-
  + \frac{1}{\eps} \Bigpars{W'(v_\eps^-) - g \sqrt{2 W(v_\eps^-)}}
  \le 0
\]
in $\Lambda_\eps^-$.

\smallParagraph{Case 2: $(x,t) \notin \Lambda_\eps^-$.}
Here, the function $v_\eps^-$ is given by $v_\eps = \pm 1 - c_2 \eps^3
\abs{\log \eps}^2$, which implies $\partial_t v_\eps^- = 0$, $\Delta v_\eps^- = 0$,
\begin{gather*}
\begin{aligned}
  \frac{1}{\eps^2} W'(v_\eps^-)
    &= \frac{1}{\eps^2} \bigpars{W'(\pm 1)
    - W''(\pm 1) c_2 \eps^3 \abs{\log \eps}^2 + O(\eps^3)} \\
    &= -W''(\pm 1) c_2 \eps \abs{\log \eps}^2 + O(\eps),
\end{aligned}
\\
\begin{aligned}
  \frac{1}{\eps} \sqrt{2 W(v_\eps^-)} g
    &= \frac{1}{\eps} \bigpars{\sqrt{2 W(\pm 1)} g
    - (\sqrt{2 W})'(\pm 1) g c_2 \eps^3 \abs{\log \eps}^2 + O(\eps^3)} \\
    &= O(\eps^2 \abs{\log \eps}^2).
\end{aligned}
\end{gather*}
Noticing that $W''(\pm 1) > 0$ and choosing $c_2$ large enough
guarantees that in $\R^d \setminus \Lambda_\eps^-$,
\[
  \partial_t v_\eps^- - \Delta v_\eps^-
  + \frac{1}{\eps} \Bigpars{W'(v_\eps^-)
  - g \sqrt{2 W(v_\eps^-)}} \le 0.
\]

To conclude, we apply the comparison principle of lemma
\ref{lem:comparison} and discover
\[
  \forall (x,t) \in \R^d \times [0,T], \quad
  v_\eps^-(x,t) \le u_\eps(x,t).
\]

A similar argument can be applied to
\[
  v_\eps^+ =
  \begin{cases}
    q_\eps(y) + \eps^2 (\bar{h} - \nabla d \cdot \nabla g) \xi_\eps(y)
      + c_2 \eps^3 \abs{\log \eps}^2
      & \text{in } \Lambda_\eps^+, \\
    -1 + c_2 \eps^3 \abs{\log \eps}^2
      & \text{in } \set{d_\eps^+ \ge 2 \delta \eps \abs{\log \eps}}, \\
    +1 + c_2 \eps^3 \abs{\log \eps}^2
      & \text{in } \set{d_\eps^+ \le -2 \delta \eps \abs{\log \eps}},
  \end{cases}
\]
with $y = \frac{d_\eps^+}{\eps}$ and
\[
  d_\eps^+(x,t) = d(x,t) - c_1(t) \eps^2 \abs{\log \eps}^2,
\]
to show that $v_\eps^+$ is a supersolution to
\eqref{eq:ac_forcing_new}:
\[
  \forall (x,t) \in \R^d \times [0,T], \quad
  v_\eps^+(x,t) \ge u_\eps(x,t).
\]

\paragraph{Proof of the theorem.}

\begin{proof}[Proof of theorem \ref{th:conv_forcing}]
We choose $\eps_0$ so that \eqref{eq:incl_Lambda} holds. Let $t$ in
$[0,T]$ and $x$ in $\partial \Omega_\eps(t)$ be given. We first show
that $x$ is in $\Lambda(t)$. Indeed, $u_\eps(x,t) = 0$ and
\begin{equation}
\label{eq:comp_ueps_veps}
  v_\eps^-(x,t) \le u_\eps(x,t) = 0 \le v^+_\eps(x,t).
\end{equation}
Assume that $x \not \in \Lambda(t)$. As $\Lambda_\eps^\pm(t) \subseteq
\Lambda(t)$, we have $x \not \in \Lambda_\eps^\pm(t)$, and thus, for
$\eps$ sufficiently small, we deduce that $v_\eps^-(x,t)$ and
$v_\eps^+(x,t)$ have the same sign. This contradicts
\eqref{eq:comp_ueps_veps}, and we conclude that $x \in \Lambda(t)$. We
then notice that
\[
  v^-_\eps(x,t) = q_\eps\biggpars{\frac{d_\eps^-(x,t)}{\eps}} + O(\eps^2)
  \le 0
\]
because
\[
 \eps^2 (\bar{h} - \nabla d \cdot \nabla g) \xi_\eps(y) = O(\eps^2),
\]
and hence
\[
  q_\eps\biggpars{\frac{d_\eps^-(x,t)}{\eps}} \le O(\eps^2).
\]
As $q'(0) = -1$, we get that
\[
  \frac{d^-_\eps(x,t)}{\eps} \ge O(\eps^2),
\]
which shows that $d(x,t) \ge O(\eps^2 \abs{\log \eps}^2)$. In a similar
way, noticing that
\[
  v^+_\eps(x,t) = q_\eps\biggpars{\frac{d_\eps^+(x,t)}{\eps}}
  + O(\eps^2) \ge 0,
\]
we get that
\[
  \frac{d^+_\eps(x,t)}{\eps} \le O(\eps^2),
\]
and $d(x,t) \le O(\eps^2 \abs{\log \eps}^2)$. We conclude that
\[
  \abs{d(x,t)} \le O(\eps^2 \abs{\log \eps}^2),
\]
and \eqref{eq:inclusion_th} is proved.
\end{proof}

\section{Application to mean curvature flow with conservation of the volume}
\label{sec:conserv_vol}

In this section, we compare two phase field models for the approximation
of motion by mean curvature with conservation of the volume:
\begin{equation}
\label{eq:mc_conserved}
  V_n = \kappa - \fint_{\partial \Omega(t)} \kappa \dsigma.
\end{equation}
As explained in the introduction, this motion is usually approximated by
the following phase field equation (see \cite{Bronsard1997}):
\begin{equation}
\label{eq:classic_model}
  \partial_t u
  = \Delta u - \frac{1}{\eps^2} W'(u)
  + \frac{1}{\eps^2} \fint_Q W'(u) \dx.
\end{equation}
The last term in this equation can be understood as a Lagrange multiplier
for the mass constraint
\[
  \frac{d}{dt} \int_Q u \dx = 0.
\]
(Note that in this section, the potential $W$ we consider has its wells
at $0$ and $1$.) In the sequel, we compare this equation to
\begin{equation}
\label{eq:new_model}
  \partial_t u = \Delta u - \frac{1}{\eps^2} W'(u)
    + \frac{1}{\eps^2} \frac{\sqrt{2 W(u)}}{\int_{\R^d} \sqrt{2 W(u)} \dx}
    \int_{\R^d} W'(u) \dx.
\end{equation}
derived along the same lines as \eqref{eq:ac_forcing_new}. The form
of the last term is again related to conservation of mass, since the
volume average of the right-hand side is easily seen to vanish.

There is no general proof of convergence of solutions of
\eqref{eq:classic_model} and \eqref{eq:new_model} to the motion
\eqref{eq:mc_conserved}. However, \eqref{eq:classic_model} is commonly
used in computations. The numerical experiments presented further
show that \eqref{eq:new_model} conserves volume with a higher degree
of accuracy than \eqref{eq:classic_model}. Our aim in this section is
to try so substantiate this claim, although our arguments are formal.

In both cases, the last term could be interpreted as a forcing term, by
setting
\[
  g_\eps(t)
  = \frac{1}{\eps c_W} \fint_Q W'(u) \dx
\]
in the first model and
\[
  \tilde{g}_\eps(t) = \frac{1}{\eps}
  \frac{ \int_{\R^d} W'(u) \dx }{ \int_{\R^d} \sqrt{2 W(u)} \dx }
\]
in the second. Formally, one recovers the expressions of
\eqref{eq:ac_forcing} and \eqref{eq:ac_forcing_new}. However the
forcing terms here depend on the solutions of \eqref{eq:classic_model}
and \eqref{eq:new_model}. Assuming that one can generalize the results
of section \ref{sec:formal_analysis} (notwithstanding this dependence
of $g_\eps$ and $\tilde{g}_\eps$), we expect solutions $u_\eps$ and
$\tilde{u}_\eps$ of \eqref{eq:classic_model} and \eqref{eq:new_model}
to have the following asymptotic behavior:
\begin{align}
\label{eqn:asymp_accc}
  u_\eps(x,t)
    &= q\biggpars{\frac{d\bigpars{x, \Omega_\eps(t)}}{\eps}}
    + \eps g_\eps(t) \eta\biggpars{\frac{d\bigpars{x, \Omega_\eps(t)}}{\eps}}
    + O(\eps^2),  \\
\label{eqn:asymp_accn}
  \tilde{u}_\eps(x,t)
    &= q\biggpars{\frac{d\bigpars{x, \tilde{\Omega}_\eps(t)}}{\eps}}
    + \eps^2 h(x,t) \xi \biggpars{\frac{d\bigpars{x, \tilde{\Omega}_\eps(t)}}{\eps}}
    + O(\eps^3),
\end{align}
where $\Omega_\eps(t)$ (resp. $\tilde{\Omega}_\eps(t)$) denotes the set
contained inside the level line $\set{ u_\eps(x,t) = \frac{1}{2} }$
(resp. $\set{ \tilde{u}_\eps(x,t) = \frac{1}{2} }$), and $q$, $\eta$,
$\xi$ are the profiles defined in \eqref{eq:dev_asymp_u},
\eqref{eq:def_eta} and \eqref{eq:def_xi}. We note that these profiles
only depend on the choice of the potential $W$. Following
\eqref{eq:dev_u_complete}, we see that as $\tilde{g}_\eps$ does not
depend on $x$, only $h$ appears in the term of order $2$ of
$\tilde{u}_\eps$.

We first establish the connection between the mass $\int_Q u_\eps \dx $
(respectively $\int_Q \tilde{u}_\eps \dx $) and the volume
$\abs{\Omega_\eps(t)}$ (respectively $\abs{\tilde{\Omega}_\eps(t)}$).

\begin{proposition}
\label{prop:volume}
Let $E$ be a regular bounded domain of $\R^d$, and let
\[
  v_\eps(x) = q\biggpars{\frac{d(x, E)}{\eps}}.
\]
Assume that $q$ is symmetric, i.e. $q(s) = 1 - q(-s)$,
 and that $q$ decays exponentially to $0$ as $s \to + \infty$. Then
\[
  \abs{E} = \int_{\R^d} v_\eps \dx + O(\eps^2).
\]
\end{proposition}

\begin{proof}
Using the co-area formula,
\begin{align*}
  \int_{\R^d} v_\eps \dx
  &= \int_{\R^d} q \biggpars{\frac{d(x,E)}{\eps}} \dx \\
  &= \int_{\R} h(s) q \biggpars{\frac{s}{\eps}} \ds \\
  &= \int_{-\infty}^0 h(s) \ds
    + \int_{-\infty}^0 h(s)
    \biggpars{ q\biggpars{\frac{s}{\eps}} - 1 } \ds
    + \int_0^{+\infty} h(s) q\biggpars{\frac{s}{\eps}} \ds \\
  &= \abs{E} - \int_0^{+\infty} h(-s) q\biggpars{\frac{s}{\eps}} \ds
    + \int_0^{+\infty} h(s) q\biggpars{\frac{s}{\eps}} \ds \\
  &= \abs{E} + \int_0^{+\infty} \bigpars{ h(s) - h(-s) }
    q\biggpars{\frac{s}{\eps}} \ds \\
  &= \abs{E} + \eps \int_0^{+\infty}
    \bigpars{ h(s\eps) - h(-s\eps) } q(s) \ds
\end{align*}
where $h(s) = \abs{D \chi_{ \set{d(x,E) \le s} }}(\R^d)$ is the
perimeter of the level line $s$ of the signed distance function to $E$.
Since $E$ is smooth, one can estimate $h(s\eps) - h(-s\eps) = 2 s
\eps h'(0) + O(s^2 \eps^2)$ for $s$ in $(0, \abs{\log \eps})$.
Furthermore, since $q$ is exponentially decreasing to $0$ as $s \to
+\infty$, all the moments $\int_{s > 0} s^n q(s) \ds$ are finite.
Thus, we can estimate
\[
  \biggabs{ \int_0^{\abs{\log \eps}}
    \bigpars{ h(s\eps) - h(-s\eps) } q(s) \ds}
  \le \biggabs{ \int_0^{\abs{\log \eps}}
    \bigpars{ 2 s \eps h'(0) + C s^2 \eps^2 } q(s) \ds}
  = O(\eps).
\]
Moreover, since $h(s) \sim s^{d-1}$ as $s \to +\infty$, and since $h$
is bounded on $(-\infty, 0)$, it is easy to check that
\begin{gather*}
  \int_{\abs{\log \eps}}^{+\infty} h(s \eps) q(s) \ds
  \le C \eps^{d-1} \int_{\abs{\log \eps}}^{+\infty} s^{d-1} q(s) \ds
  = O(\eps^{d-1}), \\
  \int_{\abs{\log \eps}}^{+\infty} h(-s\eps) q(s) \ds
  \le C \int_{\abs{\log \eps}}^{+\infty} q(s) \ds
  = O(\eps).
\end{gather*}
It follows that
\[
  \int_{\R^d} v_\eps \dx = \abs{E} + O(\eps^2).
\]
\end{proof}

The result of proposition \ref{prop:volume} still holds on a fixed
bounded set $Q$ that strictly contains $E$ when $\eps$ is sufficiently
small. This again is a consequence of the exponential decay of $q$.
Recalling the asymptotic form of $u_\eps$, it follows from the above
proposition that, for the classical model \eqref{eq:classic_model},
\[
  \int_Q u_\eps \dx
  = \int_Q q \biggpars{\frac{d\bigpars{x, \Omega_\eps(t)}}{\eps}} \dx
  + \eps g_\eps \int_Q \eta
  \biggpars{\frac{d\bigpars{x, \Omega_\eps(t)}}{\eps}} \dx + O(\eps^2).
\]
In general, the term of order $\eps$ does not vanish, since
\[
  \lim_{s \to \pm\infty} \eta(s) = \frac{c_W}{W''(0)} \neq 0
\]
when $q$ is symmetric, and so
\[
  \int_Q u_\eps \dx = \abs{\Omega_\eps(t)} + O(\eps).
\]
This explains why we cannot expect the model \eqref{eq:classic_model}
to converge to the motion \eqref{eq:mc_conserved} with a better rate
than $O(\eps)$. As for the model \eqref{eq:new_model}, we have
\[
  \int_Q \tilde{u}_\eps \dx
  = \int_Q q\biggpars{\frac{d\bigpars{x, \tilde{\Omega}_\eps(t)}}{\eps}} \dx
  + \eps^2  \int_Q h \xi
  \biggpars{\frac{d\bigpars{x, \tilde{\Omega}_\eps(t)}}{\eps}} \dx + O(\eps^3),
\]
that is
\[
  \int_Q \tilde{u}_\eps \dx = \abs{\tilde{\Omega}_\eps(t)} + O(\eps^2),
\]
which presents a higher degree of accuracy on volume conservation.

We proved in the last section that solutions of
\eqref{eq:ac_forcing_new} converge as $\eps \to 0$ to motion by mean
curvature with a forcing term \eqref{eq:mc_motion_forcing}. Formally,
the phase field equation \eqref{eq:new_model} can be rewritten
\[
  \partial_t u = \Delta u - \frac{1}{\eps^2}
  \biggpars{ W'(u) - \eps \sqrt{2 W(u)} \tilde{g}_\eps }.
\]
The following property shows that, under the assumption
\eqref{eqn:asymp_accn}, $\tilde{g}_\eps$ converges to $\fint_{\partial
\Omega_\eps} \kappa \dsigma$, which is formally consistent to the
limiting motion \eqref{eq:mc_motion_forcing}.

\begin{proposition}
\label{prop:forcing2}
Let $E$ be a regular bounded domain of $\R^d$, and let
\[
  v_\eps(x) = q \biggpars{ \frac{d(x,E)}{\eps} }.
\]
Assume that $q$ is symmetric, i.e. $q(s) = 1 - q(-s)$, and that $q$
decays exponentially to $0$ as $s \to +\infty$. Then
\begin{align*}
  \frac{ \int_{\R^d} g \sqrt{2 W(v_\eps)} \dx }
  { \int_{\R^d} \sqrt{2 W(v_\eps)} \dx }
  &= \fint_{\partial E} g \dsigma + O(\eps^2), \\[2ex]
  \frac{1}{\eps} \frac{ \int_{\R^d} W'(v_\eps) \dx }
  { \int_{\R^d} \sqrt{2 W(v_\eps)} \dx }
  &= -\fint_{\partial E} \kappa \dsigma + O(\eps^2).
\end{align*}
\end{proposition}

\begin{proof}
To prove the first equality, recall that $q$ satisfies $\sqrt{2 W(q)}
= -q'$, and that $q'$ is even. Let $h \colon \R \to \R$ be a continuous
function, differentiable at $s = 0$, which grows polynomially in $s$.
Since $\int_\R q'(s) \ds = -1 $, arguing as in the proof of proposition
\ref{prop:volume},
it follows that
\begin{align*}
  \frac{1}{\eps} \int_\R
  \sqrt{2 W\biggpars{ q\biggpars{\frac{s}{\eps}} }} h(s) \ds
  &= -\int_\R \frac{1}{\eps} q'\biggpars{\frac{s}{\eps}} h(s) \ds \\
  &= -\int_\R q'(s) h(s\eps) \ds \\
  &= -\int_0^{+\infty} \bigpars{ h(s\eps) + h(-s\eps)} q'(s) \ds \\
  &= -\int_0^{3 \abs{\log \eps}} \bigpars{ h(s\eps) + h(-s\eps)} q'(s) \ds + O(\eps^2) \\
  &= -\int_0^{3 \abs{\log \eps}} \bigpars{ 2 h(0) + C s^2 \eps^2} q'(s) \ds + O(\eps^2) \\
  &= h(0) + O(\eps^2).
\end{align*}
Next, the co-area formula yields
\[
  \frac{1}{\eps} \int_{\R^d} g \sqrt{2 W(v_{\eps})} \dx
  = \frac{1}{\eps} \int_\R
  \biggpars{ \int_{d(x,E) = s } g \dsigma }
  \sqrt{2 W\biggpars{ q\biggpars{\frac{s}{\eps}} }} \ds.
\]
Since $E$ is smooth, and since the forcing term $g$ is bounded,
the function
\[
  h \colon s \mapsto \int_{d(x,E) = s} g \dsigma
\]
is continuous, differentiable at $s = 0$ and has polynomial growth
at infinity: $h(s) \sim s^{d-1}$ when $s \to +\infty$.  We can then
apply the previous estimate to obtain
\[
  \frac{1}{\eps} \int_{\R^d} g \sqrt{2 W(v_\eps)} \dx
  = \int_{d(x, E) = 0} g \dsigma + O(\eps^2)
  = \int_{\partial E} g \dsigma + O(\eps^2).
\]
We notice that the same argument with $g = 1$ leads to
\[
  \frac{1}{\eps} \int_{\R^d} \sqrt{2 W(v_\eps)} \dx
  = \int_{d(x,E)= 0} \dsigma + O(\eps^2)
  = \abs{\partial E} + O(\eps^2),
\]
so that combined with the previous equality, we obtain
\[
  \frac{\int_{\R^d} g \sqrt{2 W(v_\eps)} \dx}
  {\int_{\R^d} \sqrt{2 W(v_\eps)} \dx}
  = \fint_{\partial E} g \dsigma + O(\eps^2).
\]

Let us now prove the second equality. Recall that
\begin{align*}
  v_\eps &= q \biggpars{ \frac{d(x,E)}{\eps} }, \\
  \nabla v_\eps &= \frac{1}{\eps} q' \biggpars{ \frac{d(x,E)}{\eps} }
\nabla d(x,E), \\
  \Delta v_\eps &= \frac{1}{\eps^2} q'' \biggpars{ \frac{d(x,E)}{\eps}
}
    + \frac{1}{\eps} q'\biggpars{ \frac{d(x,E)}{\eps} } \Delta
d(x,E).
\end{align*}
As $q'' = W'(q)$, it follows that
\begin{align*}
  \frac{1}{\eps^2} \int_{\R^d} W'(v_\eps) \dx
  &= \int_{\R^d} \biggpars{\frac{1}{\eps^2} W'(v_\eps) - \Delta v_\eps} \dx \\
  &= -\int_{\R^d} \frac{1}{\eps} q'\biggpars{\frac{d(x,E)}{\eps} }
\Delta d(x,E) \dx \\
  &= \int_\R \biggpars{ \int_{d(x,E) = s} \Delta d(x,E) \dsigma }
\frac{1}{\eps} q' \biggpars{\frac{s}{\eps}} \ds.
\end{align*}
The function
\[
  s \mapsto h(s) = \int_{d(x,E) = s} \Delta d(x,E) \dsigma
\]
is not continuous on $\R$, but it is constant on a sufficiently small
neighborhood of $0$ (depending only on the topology of $E$) and grows
polynomially like $s^{d-1}$. Arguing as in the first part of the
proof, we obtain
\[
  \frac{1}{\eps^2} \int_{\R^d} W'(v_\eps) \dx
  = h(0) + O(\eps^2)
  = -\int_{\partial E} \kappa \dsigma + O(\eps^2),
\]
and
\[
  \frac{1}{\eps} \frac{ \int_{\R^d} W'(v_\eps) \dx }
  { \int_{\R^d} \sqrt{2 W(v_\eps)} \dx }
  = -\fint_{\partial E} \kappa \dsigma + O(\eps^2),
\]
which completes the proof.
\end{proof}

\begin{remark}
The first correcting term $\xi$ vanishes at infinity in the expansion
of $\tilde{u}_\eps$, and so would higher order terms. If we assume that
\eqref{eqn:asymp_accn} holds, a more careful analysis based on
proposition \ref{prop:forcing2} would show that
\[
  \frac{1}{\eps} \frac{ \int_{\R^d} W'(\tilde{u}_\eps) \dx }
  { \int_{\R^d} \sqrt{2 W(\tilde{u}_\eps)} \dx}
  = -\fint_{\partial \tilde{\Omega}_\eps} \kappa \dsigma + O(\eps^2).
\]
This heuristically justifies the use of \eqref{eq:new_model} as an
approximation to the motion \eqref{eq:mc_conserved}.
\end{remark}

\begin{remark}
We can generalize our previous argument to the case of interfaces
moving with normal velocity
\begin{equation}
\label{motiom_mc_g}
  V_n = \kappa + g - \fint_{\partial \Omega} (\kappa + g) \dsigma.
\end{equation}
The usual phase field approximation of such motions is based on
the equation
\begin{equation}
\label{modele_classique2}
  \partial_t u = \Delta u - \frac{1}{\eps^2} \biggpars{W'(u) - \eps c_W g}
  + \frac{1}{\eps^2} \fint_Q \biggpars{ W'(u) - \eps c_W g } \dx,
\end{equation}
where the last term can be understood as a Lagrange multiplier. As
explained above, one cannot expect that this model should converge to
the motion \eqref{motiom_mc_g} with a better rate than $O(\eps)$.
Generalizing our previous analysis, we may instead consider the
following modified phase field model, which should improve the accuracy:
\begin{multline}
\label{modele_elie2}
  \partial_t u = \Delta u - \frac{1}{\eps^2} \biggpars{ W'(u) - \eps g \sqrt{2W(u)} } \\
  + \frac{1}{\eps^2} \frac{\sqrt{2 W(u)}}{ \int_{\R^d} \sqrt{2 W(u)} \dx }
  \int_{\R^d} \biggpars{ W'(u) - \eps g \sqrt{2 W(u)} } \dx.
\end{multline}
\end{remark}

\section{Numerical method and simulations}
\label{sec:numeric}

In this section, we describe the numerical method we use for solving
\begin{equation}
\label{eq:gal}
\begin{dcases}
  \partial_t u = \Delta u - \frac{1}{\eps^2} F(u)
    & x \in Q \subset \R^d, \quad t \in [0,T], \\[2ex]
  u(x,0) = u_0(x),
    & x \in Q,
\end{dcases}
\end{equation}
where $F$ takes one of the following forms:
\begin{align*}
  W'_{\eps, g}(u)
    &= W'(u) - \eps c_W g, \\
  \tilde{W}'_{\eps, g}(u)
    &= W'(u) - \eps g \sqrt{2 W(u)}, \\
  W'_{\eps, g, vol}(u)
    &= W'(u) - \eps c_W g
    - \fint_Q \biggpars{W'(u) - \eps c_W g} \dx, \\
  \tilde{W}'_{\eps, g, vol}(u)
    &= W'(u) - \eps g \sqrt{2 W(u)} \\
    &\qquad - \frac{\sqrt{2 W(u)}}{\int_Q \sqrt{2 W(u)} \dx}
    \int_Q \biggpars{W'(u) - \eps g \sqrt{2 W(u)}} \dx.
\end{align*}
The first form corresponds to the Allen--Cahn equation with a
forcing term $g$. The second form corresponds to the modified
approximation introduced in section \ref{sec:new_model}. Forms 3
and 4 are the respective forms when the volume is conserved (see
\eqref{modele_classique2} and \eqref{modele_elie2}). We assume that
\[
  u_0 = q \biggpars{\frac{d(x, \partial \Omega_0)}{\eps}},
\]
where $\Omega_0$ is a smooth bounded set of $\R^d$ strictly contained
in the fixed box $Q = {[-\frac{1}{2}, \frac{1}{2}]}^d$, with $d
= 2$ or $3$. We assume also that during the evolution the sets
$\Omega_\eps(t)$ remain within $Q$, so that we may impose periodic
boundary conditions on $\partial Q$ to the solutions of \eqref{eq:gal}.

\subsection{Numerical scheme}

Equation \eqref{eq:gal} is numerically approximated via a splitting
method between the diffusion and reaction terms.  We take advantage
of the periodicity to treat the diffusion part of the operator in
the Fourier space. More precisely, the value $u_\eps(x, t_n)$ at time
$t_n = t_0 + n \Delta t$ is approximated by
\[
  u_\eps^P(x, t_n)
  = \sum_{\max\limits_{1 \le k \le d} \abs{p_k} \le P}
    u_{\eps, p}(t_n) \exp (2 i \pi p \cdot x).
\]
In a first step, we set
\[
  u_\eps^P \biggpars{x, t_n + \frac{1}{2}}
  = \sum_{\max\limits_{1 \le k \le d} \abs{p_k} \le P}
    u_{\eps, p}\biggpars{t_n + \frac{1}{2}} \exp (2 i \pi p \cdot x),
\]
with
\[
  u_{\eps, p}\biggpars{t_n + \frac{1}{2}}
  = u_{\eps, p}(t_n) \exp (-4 \pi^2 \Delta t \abs{p}^2).
\]
We then add the reaction term:
\[
  u_\eps^P(x, t_n+1)
  = u_\eps^P\biggpars{t_n + \frac{1}{2}}
  - \frac{\Delta t}{\eps^2} F \biggpars{ u_\eps^P \biggpars{t_n + \frac{1}{2}} }.
\]
In practice, the first step is performed via a fast Fourier transform,
with a computational cost of $O(P^d \log P)$. The corresponding
numerical scheme turns out to be $L^\infty$-stable for the standard
Allen--Cahn equation with no forcing term, under the condition
\[
  \delta t \le M \eps^2,
\]
where $M = \bigpars{ \sup_{t \in [0,1]} W''(t) }^{-1}$. It can
be shown that this condition is also sufficient for the modified
potential $\tilde{W}_{\eps, g}$. We impose this constraint in the
following computations for all the choices of $F$. We use the double
well potential $W(s) = \frac{1}{2} s^2(1 - s)^2$, and $P$ represents
the number of Fourier modes in each dimension.

\subsection{Numerical tests}

\paragraph{Convergence test with no forcing term.}
This test illustrates the convergence of our numerical scheme when
we consider the equation
\[
  \partial_t u = \Delta u - \frac{1}{\eps^2} W'(u),
\]
with no forcing term, nor volume conservation. The initial set
$\Omega_0$ is taken as a circle of radius $R_0 = 0.25$. It should
evolve as a circle, with radius $R(t) = \sqrt{{R_0}^2 - 2t}$, that
decreases to a point at the extinction time $t_{ext} = \frac{1}{2}
R_0^2$. Figure \ref{fig:allen-cahn1} represents $\Omega(t)$ at
different times, for the choice of parameters $P = 2^8$, $\Delta
t = 1/P^2$ and $\eps = 2/P$. Figure \ref{fig:allen-cahn2} shows
the error between calculated and theoretical extinction times for
different values of $\eps$, in logarithmic scale. The error behaves
like $O(\eps^2 \abs{\log \eps}^2)$ as expected. This indicates that,
with this choice of parameters, the error due to our numerical scheme
is negligible compared to the `modeling' error due to the approximation
of the motion by the phase field equation.

\begin{figure}[htbp]
\centering
\includegraphics[width=0.6\linewidth]{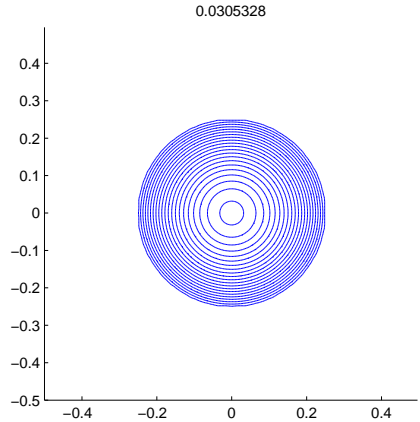}
\caption{Mean curvature flow of a circle: level set $\set{u_\eps(x,t) =
\frac{1}{2}}$ for different times.}
\label{fig:allen-cahn1}
\end{figure}

\begin{figure}[htbp]
\centering
\includegraphics[width=0.45\linewidth]{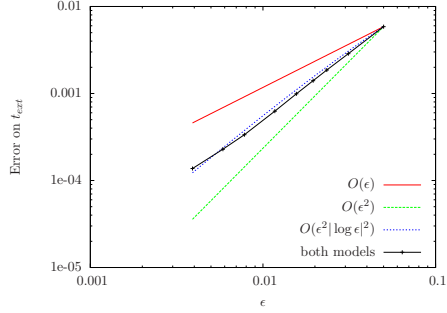}
\caption{Mean curvature flow of a circle: error on the extinction time
for different values of $\eps$ (logarithmic scale).}
\label{fig:allen-cahn2}
\end{figure}

\paragraph{Convergence test with a constant forcing term.}
Here we compare the two phase field models \eqref{eq:ac_forcing}
and \eqref{eq:ac_forcing_new} as approximations to the motion
\eqref{eq:mc_motion_forcing}. Theoretically, both give an approximation
order of $O(\eps^2 \abs{\log \eps}^2)$.  We compare the numerical
solutions in the simple case where the forcing term is a constant: $g
= C_g$.  The initial condition $\Omega_0$ is a circle of radius $R_0$.
During the evolution, $\Omega(t)$ also remains circular, and its radius
$R$ satisfies
\[
  \frac{dR}{dt} = -\frac{1}{R} + C_g.
\]
Assuming that $C_g < 1 / R_0$, $\Omega(t)$ decreases to a point,
with extinction at the time
\[
  t_{ext} = -\frac{1}{C_g}
  \biggpars{ \frac{1}{C_g} \ln (1 - C_g R_0) + R_0 }.
\]
We represent on figure \ref{fig:allen-cahn-forcing} the error on
the extinction time for different values of $\eps$, in logarithmic
scale. We choose $C_g = 2$ and $C_g = -2$ respectively. Both models
give comparable results, and as expected by the theory, we again
observe a $O(\eps^2 \abs{\log \eps}^2)$ error.

\begin{figure}[htbp]
\hspace*{\fill}
\includegraphics[width=0.45\linewidth]{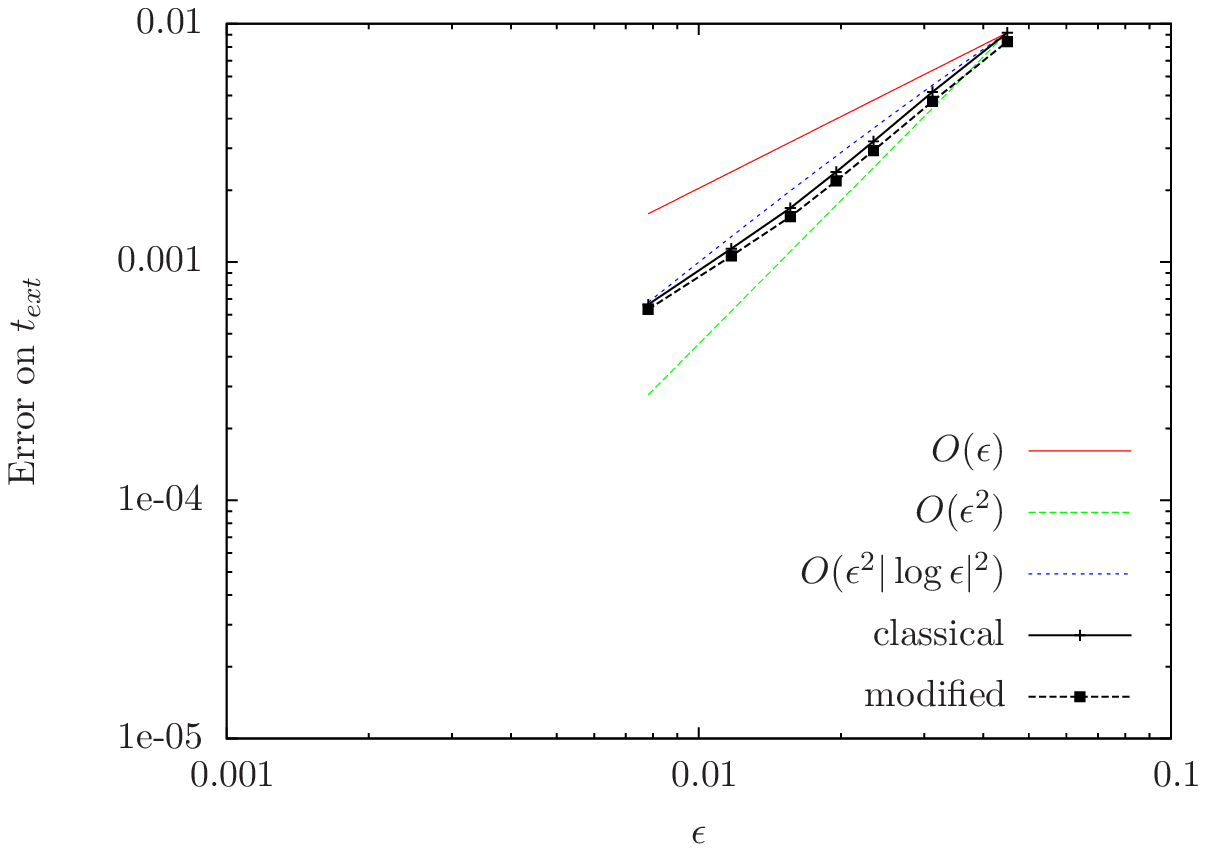}
\hfill
\includegraphics[width=0.45\linewidth]{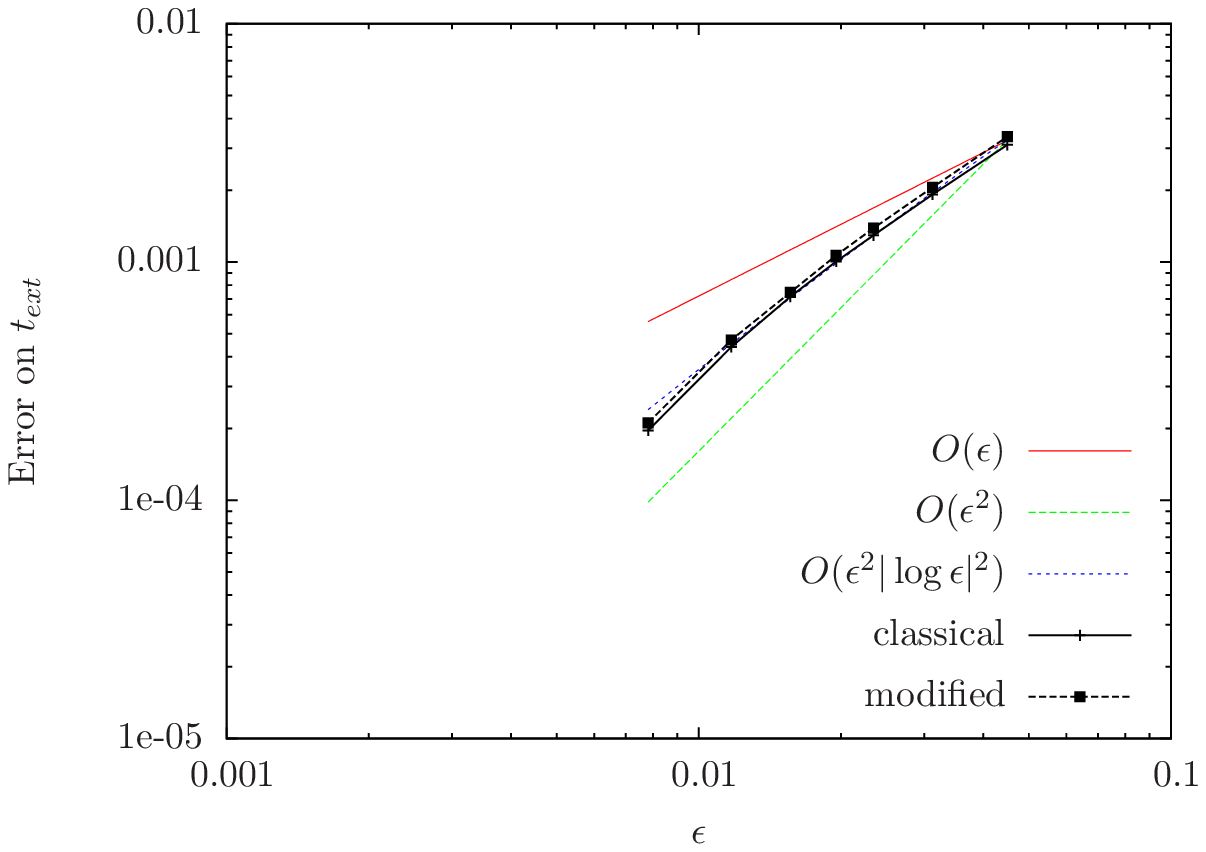}
\hspace*{\fill}
\caption{Mean curvature flow of a circle with a constant forcing term
$C_g$: error on the extinction time for different values of $\eps$
(logarithmic scale). Left: $C_g=2$. Right: $C_g=-2$.}
\label{fig:allen-cahn-forcing}
\end{figure}

\paragraph{Conservation of the volume with no forcing term.}
Here the initial configuration $\Omega_0$ is the union of two disjoints
circles of respective radii $r_0$ and $R_0$, with $r_0 < R_0$. As it
evolves by conserved mean curvature flow \eqref{eq:mc_conserved},
$\Omega$ remains the union of two circles, with radii $r$ and $R$
solutions of
\[
\begin{dcases}
  \frac{d r}{dt} = -\frac{1}{r} + \frac{2}{r + R}, \\[2ex]
  \frac{d R}{dt} = -\frac{1}{R} + \frac{2}{r + R}.
\end{dcases}
\]
It is easy to check that the smallest circle decreases and disappears
at extinction time
\[
  t_{ext} = -\frac{r_0R_0}{2} + \frac{R_0^2 + r_0^2}{4} \ln\biggpars{1
  + \frac{2r_0R_0}{(R_0 - r_0)^2}}.
\]
Meanwhile, the radius of the initially larger circle grows to a
maximal value
\[
  R_* = \sqrt{r_0^2 + R_0^2}
\]
at extinction time.  We presents results for $r_0 = 0.1$, $R_0 = 0.15$,
$t_{ext} = 0.0133$, and for the choice of numerical parameters $P =
2^8$, $\Delta t = 2^{-16}$. The evolution of $r$ and $R$ is plotted on
figure \ref{fig:test1_conserve_bicercle1} for both models
\eqref{eq:classic_model} and \eqref{eq:new_model} and for different
choices of $\eps$:. Figure \ref{fig:test1_conserve_bicercle3} depicts
the error on extinction time in logarithmic scale. The graph clearly
shows that the error on extinction time is of order $\eps$ for the
classical model, while it scales like $\eps^2$ for the modified model
\eqref{eq:new_model}.

\begin{figure}[htbp]
\hspace*{\fill}
\includegraphics[width=0.4\linewidth]{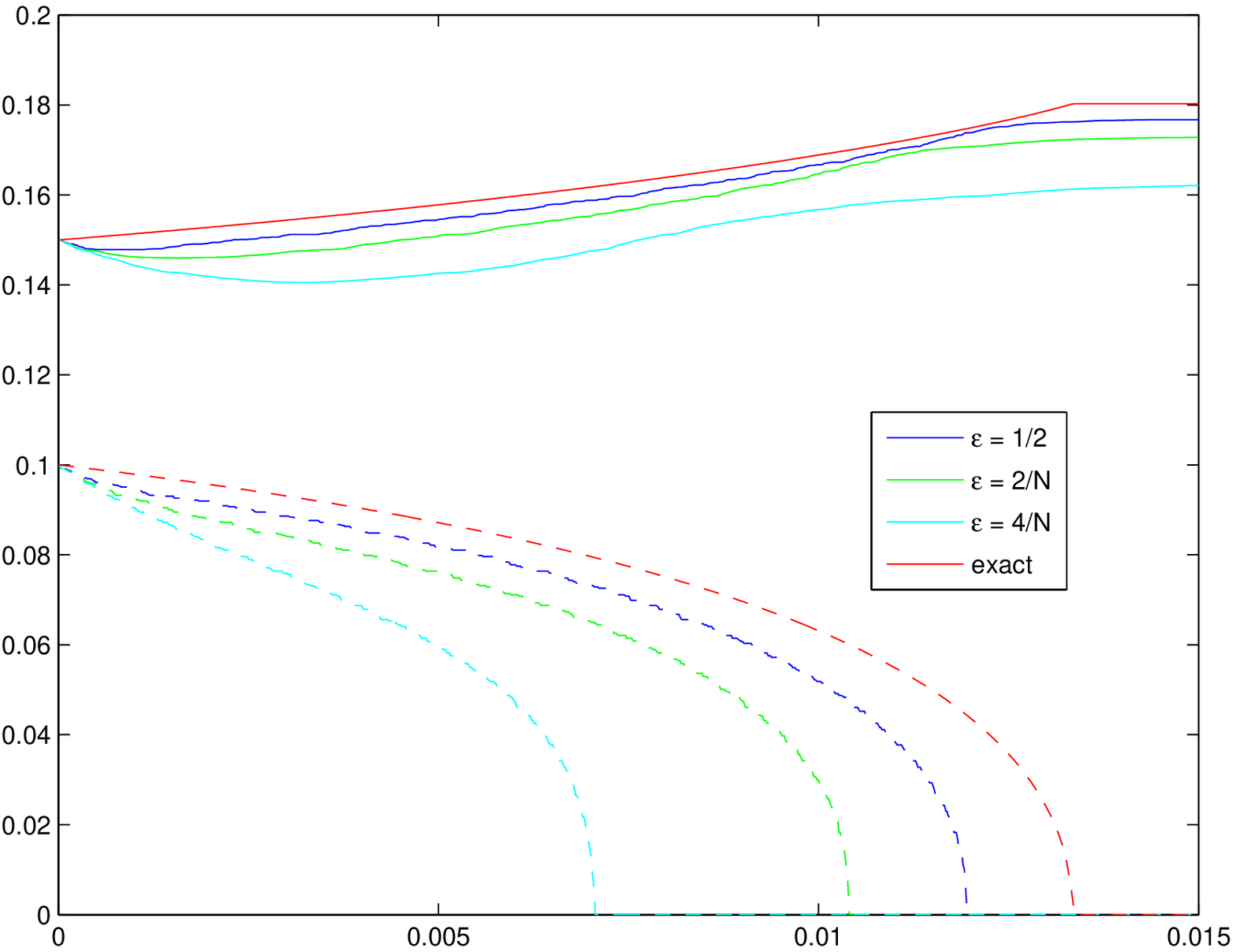}
\hfill
\includegraphics[width=0.4\linewidth]{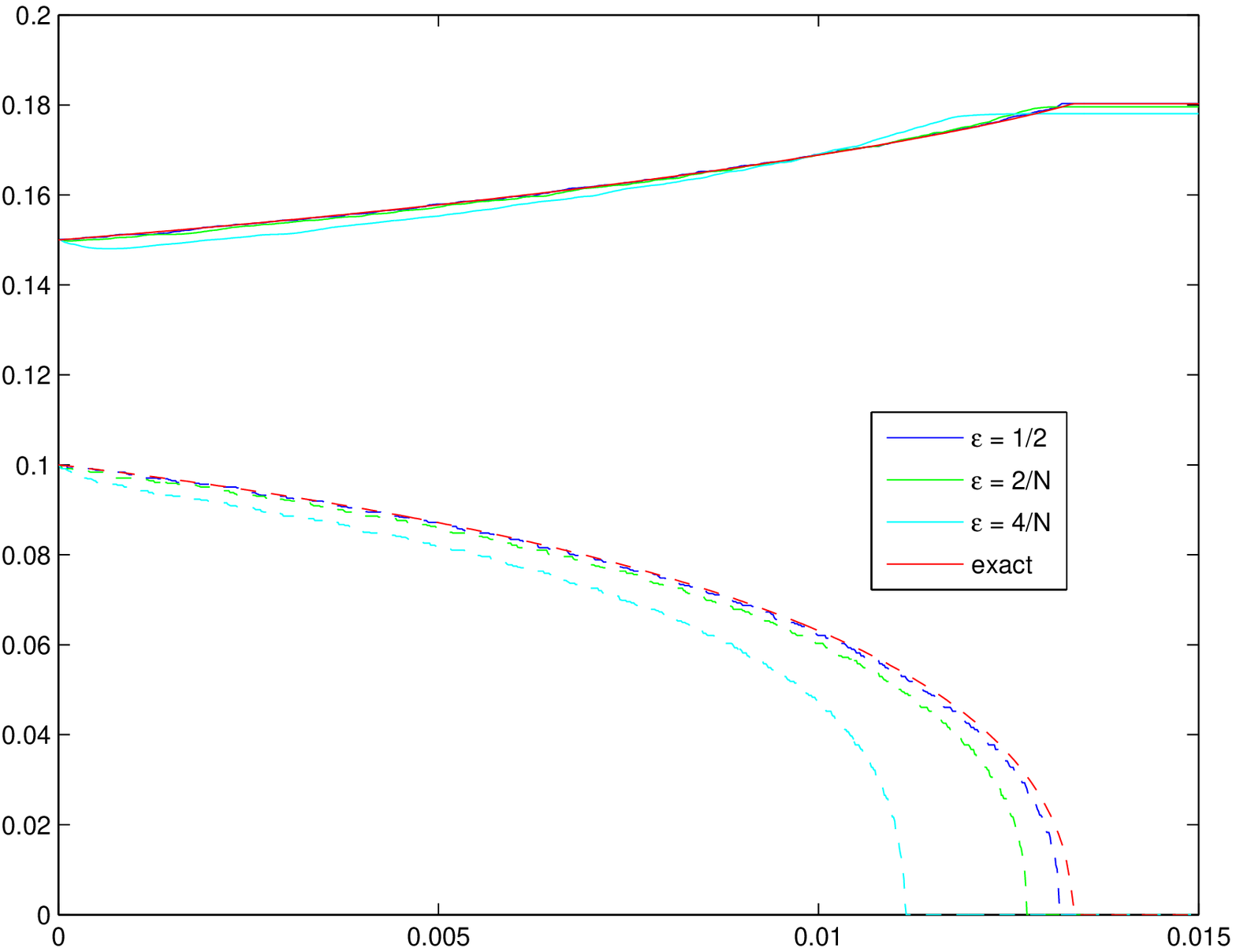}
\hspace*{\fill}
\caption{Conserved mean curvature flow of two disjoint circles:
evolution of the radii against time, for different values of $\eps$.
Left: classical model \eqref{eq:classic_model}. Right: modified model
\eqref{eq:new_model}.}
\label{fig:test1_conserve_bicercle1}
\end{figure}

\begin{figure}[htbp]
\centering
\includegraphics[width=0.45\linewidth]{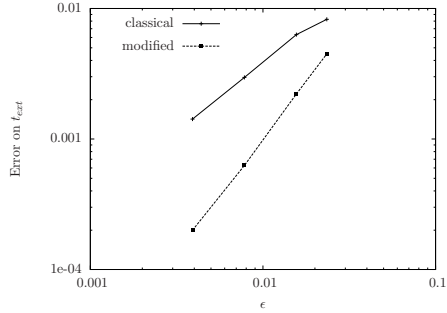}
\caption{Conserved mean curvature flow of two disjoint circles: error on
the extinction time against $\eps$ for the two models
\eqref{eq:classic_model} and \eqref{eq:new_model}.}
\label{fig:test1_conserve_bicercle3}
\end{figure}

\paragraph{Conservation of the volume with a non-zero forcing term.}
Volume losses may become important when approximating forced mean
curvature motion with the classical phase field model
\eqref{eq:classic_model}. The purpose of this test is to illustrate this
point. We choose $g$ to be an isotropic forcing term: $g(x) = c_g \cos(8
\pi \abs{x})$. The initial configuration $\Omega_0$ is the circle of
radius $R_0 = 0.25$ centered at $0$. It should remain stationary (i.e.
$\Omega(t) = \Omega_0$ for all $t$) whatever the value of the constant
$c_g$. Figure \ref{fig:test_conserve_forcing} represents the computed
evolutions using respectively \eqref{modele_classique2} and
\eqref{modele_elie2}. The numerical parameters are $P = 2^8$, $\eps =
2/P$ and $\Delta t=1/P^2$. Clearly, the value of $c_g$ has a significant
impact on the results when using \eqref{modele_classique2}.
Comparatively, the choice of $c_g$ as a negligible impact on the
evolutions computed with \eqref{modele_elie2}. This confirms the
arguments developed in section \ref{sec:conserv_vol}.

\begin{figure}[htbp]
\hspace*{\fill}
\includegraphics[width=0.4\linewidth]{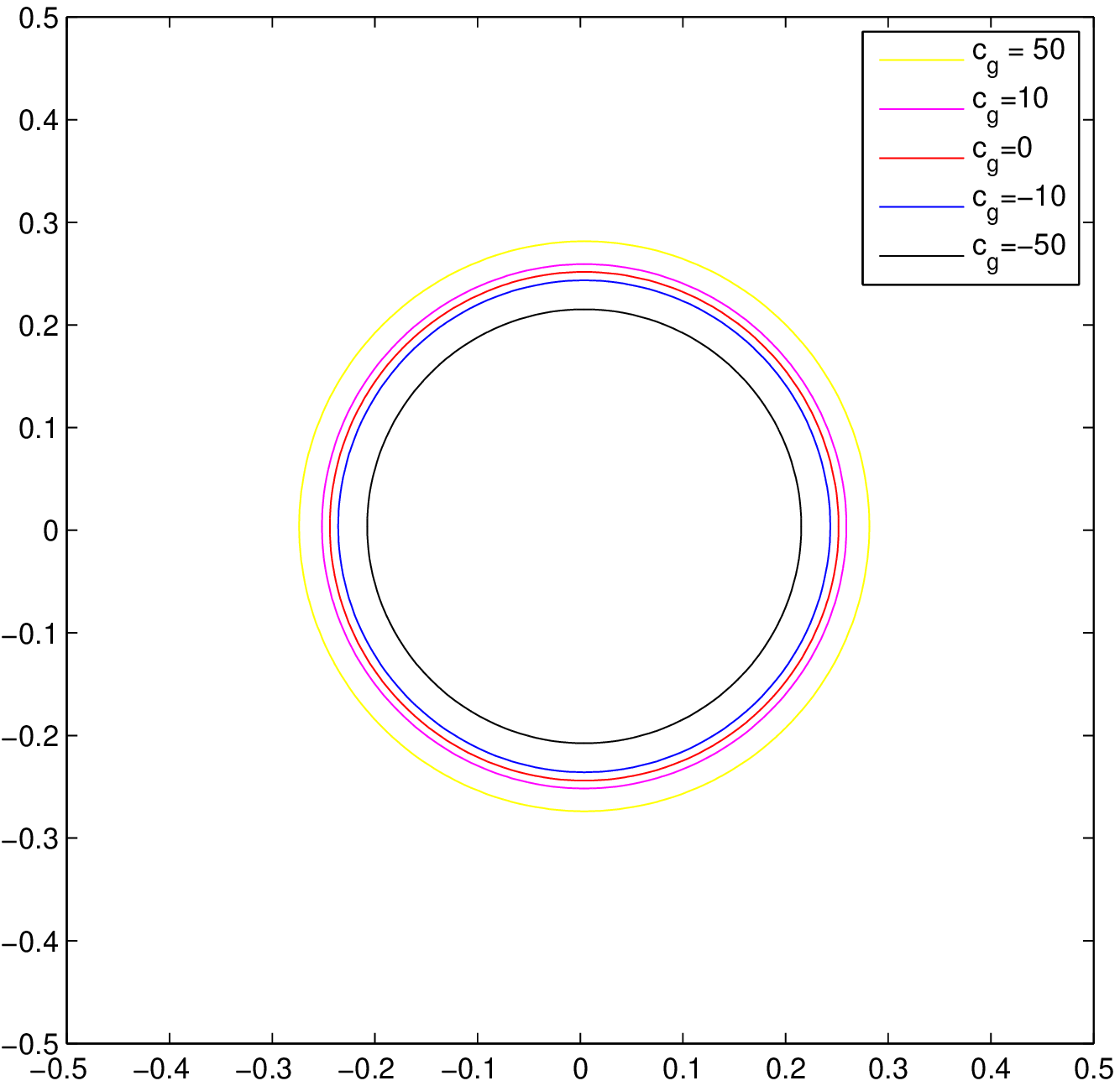}
\hfill
\includegraphics[width=0.4\linewidth]{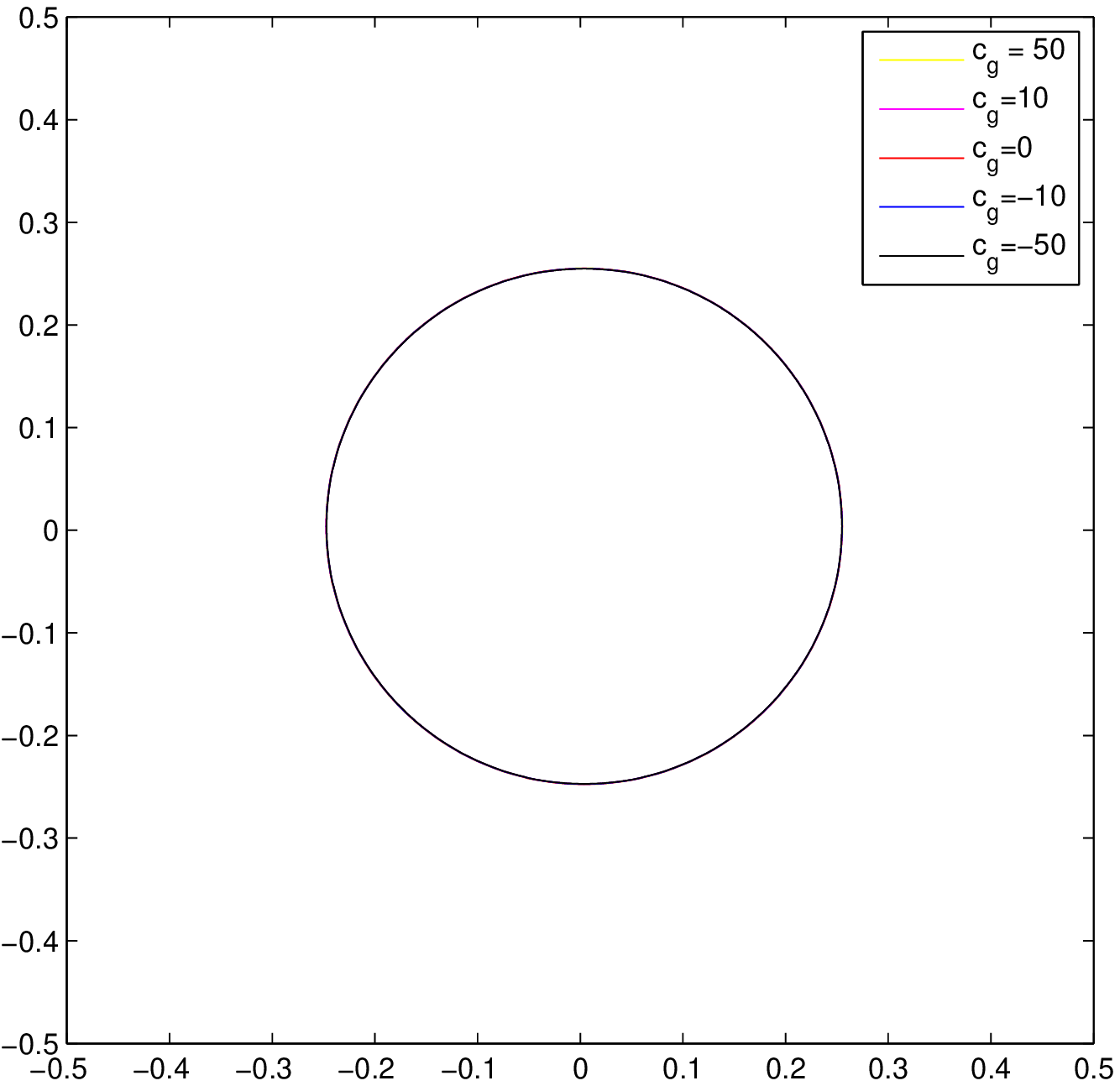}
\hspace*{\fill}
\caption{Conserved mean curvature flow of a circle with an additional
isotropic forcing term $g(x) = c_g \cos(8 \pi \abs{x})$: stationary
shape for different choices of $c_g$. Left: classical model
\eqref{modele_classique2}. Right: modified model \eqref{modele_elie2}.}
\label{fig:test_conserve_forcing}
\end{figure}

\paragraph{An example in 3D.}
Here we illustrate the benefits of our approach on a classical
three-dimensional example: the evolution of a torus with conservation of
the volume and no additional forcing term. This example provides a good
test case: because of the high values taken by the mean curvature,
standard approaches may fail to reproduce the motion correctly. One also
need to handle the topological change when we move from a toric shape to
a spherical one.

We clearly observe on figure \ref{fig:torus} that the classical model
\eqref{modele_classique2} leads to significant volume losses compared to
our modified model \eqref{modele_elie2}. We plot on figure
\ref{fig:vol_torus} the volume against time for both approaches. The
volume error goes up to 30\% for the classical model, whereas it is
always strictly below 5\% for ours. We notice that, in both cases, the
error decreases in the second part of the evolution. Indeed, it is clear
that the numerical error is maximal when the average mean curvature is
maximal; when the topological change occurs, the average mean curvature
instantly jumps to a smaller value, as the points where the mean
curvature is the highest just disappear from the surface.

\begin{figure}[htbp]
\hspace*{\fill}
\includegraphics[width=0.2\linewidth]{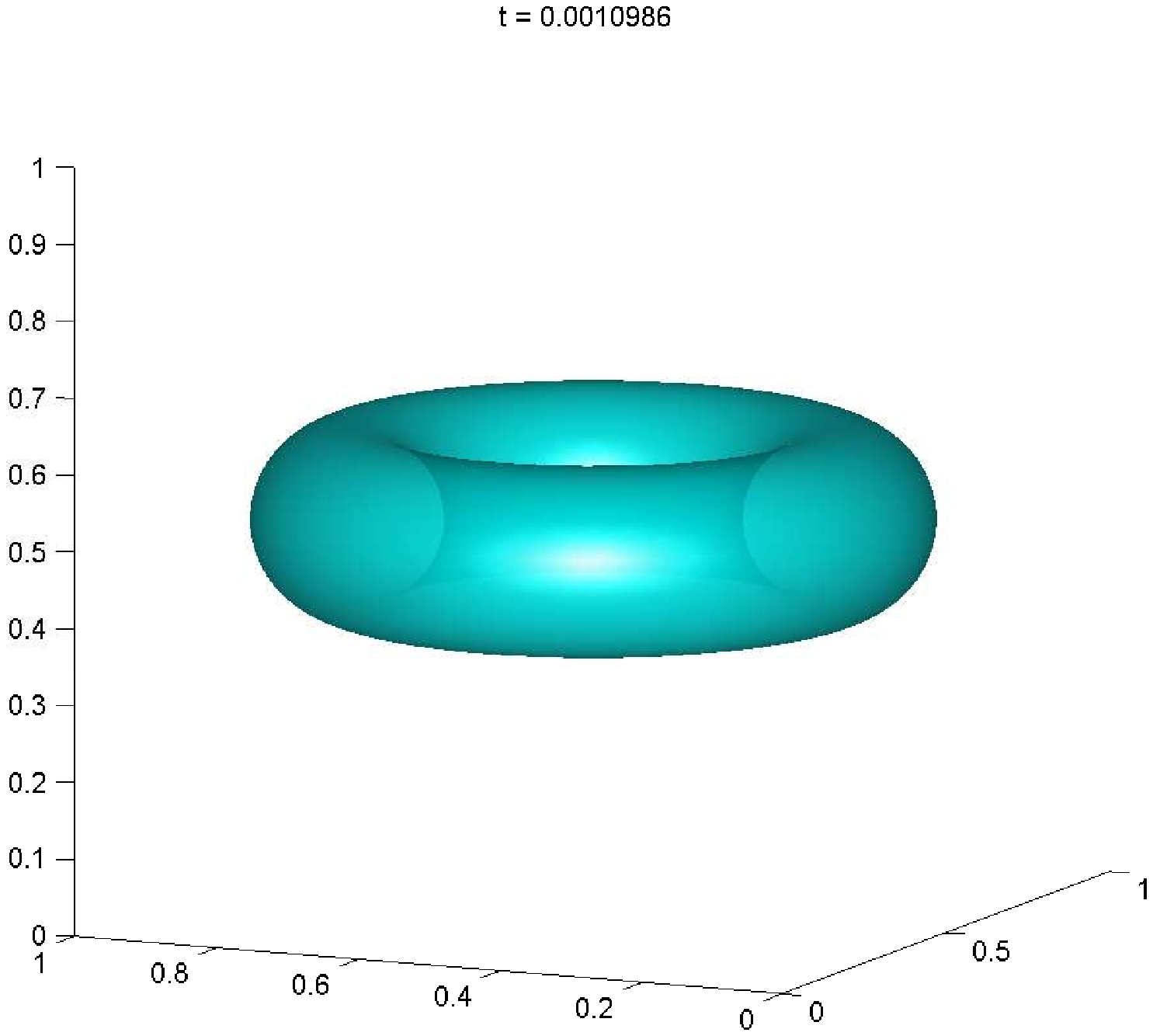}\hfill
\includegraphics[width=0.2\linewidth]{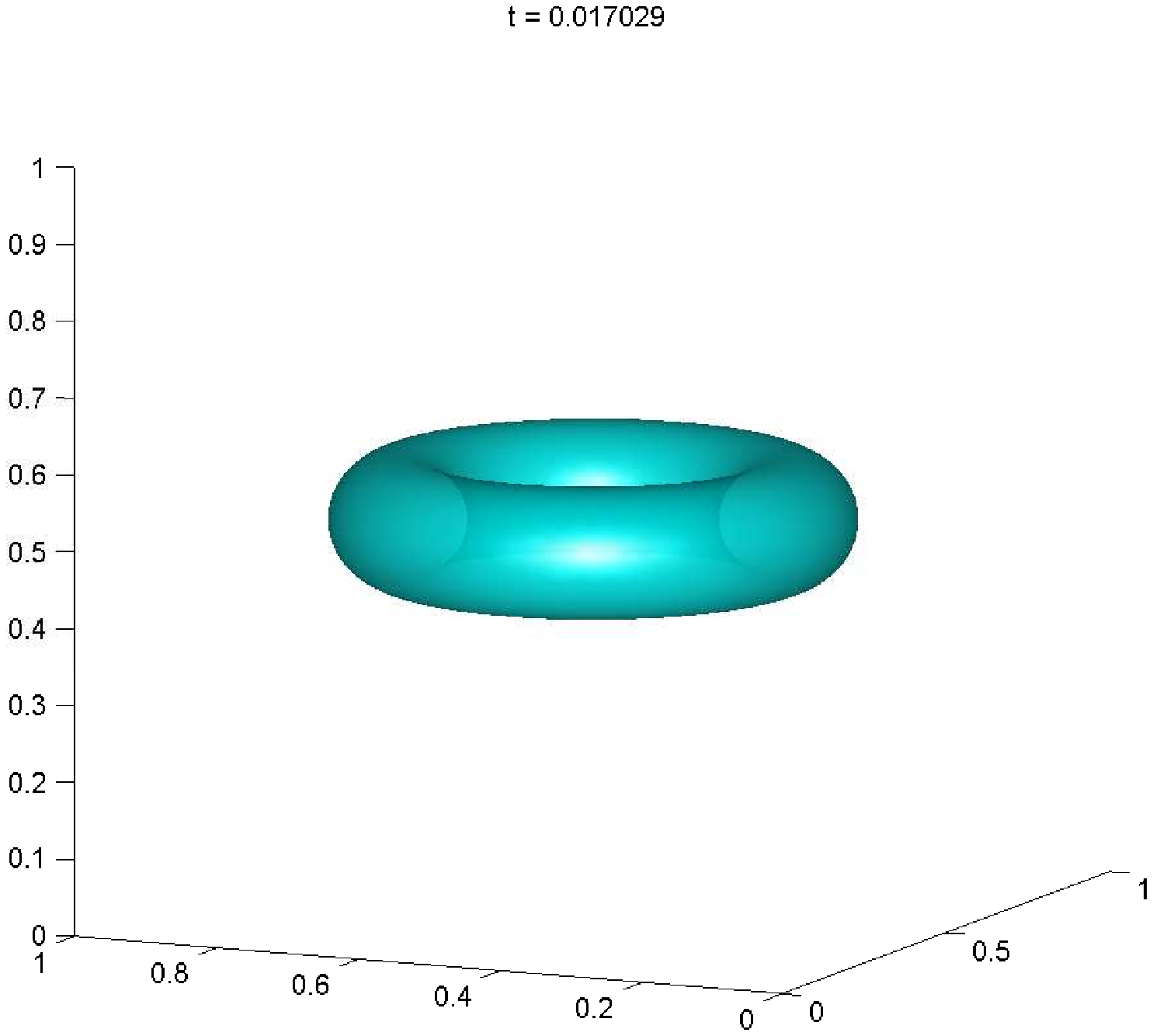}\hfill
\includegraphics[width=0.2\linewidth]{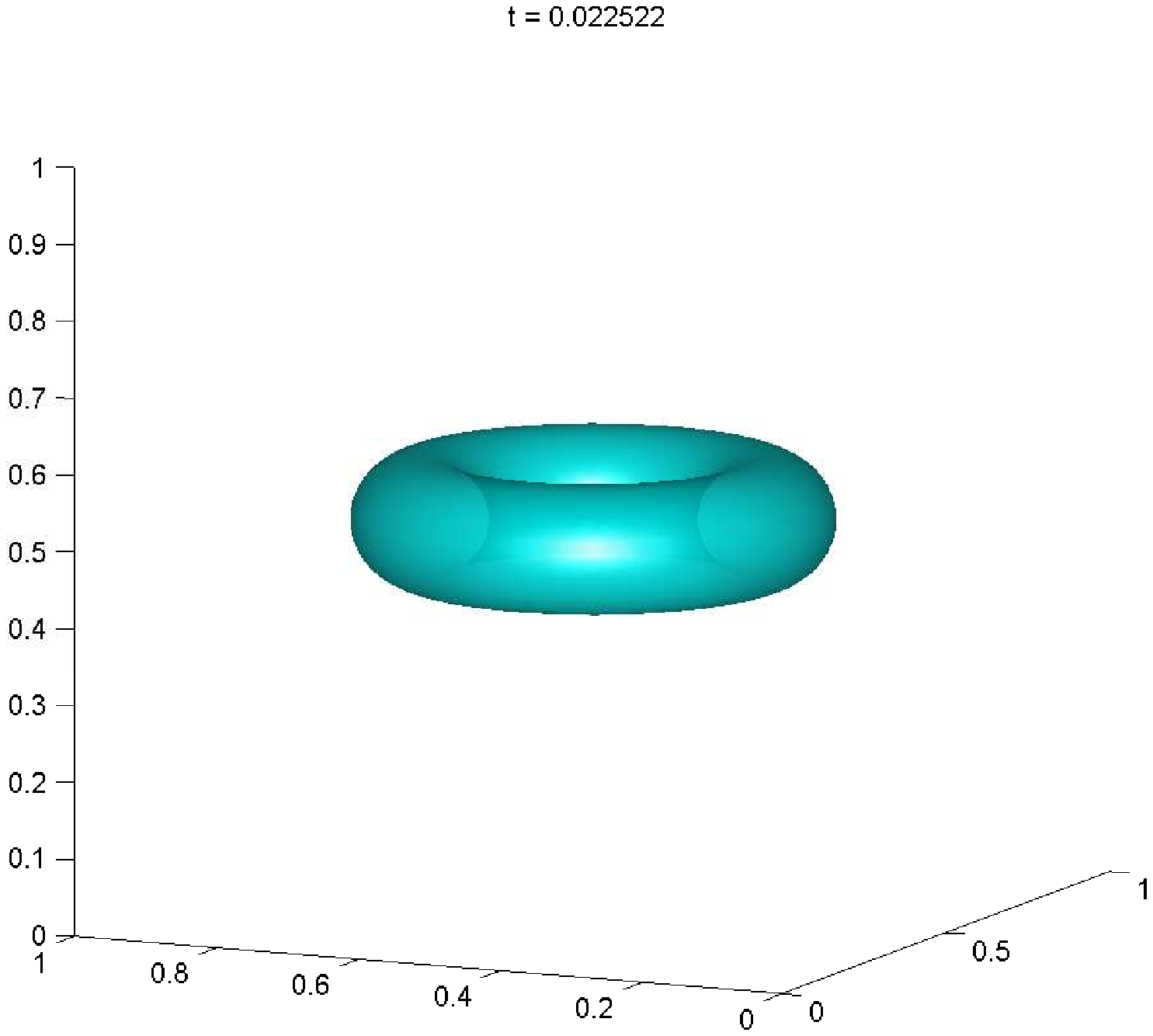}\hfill
\includegraphics[width=0.2\linewidth]{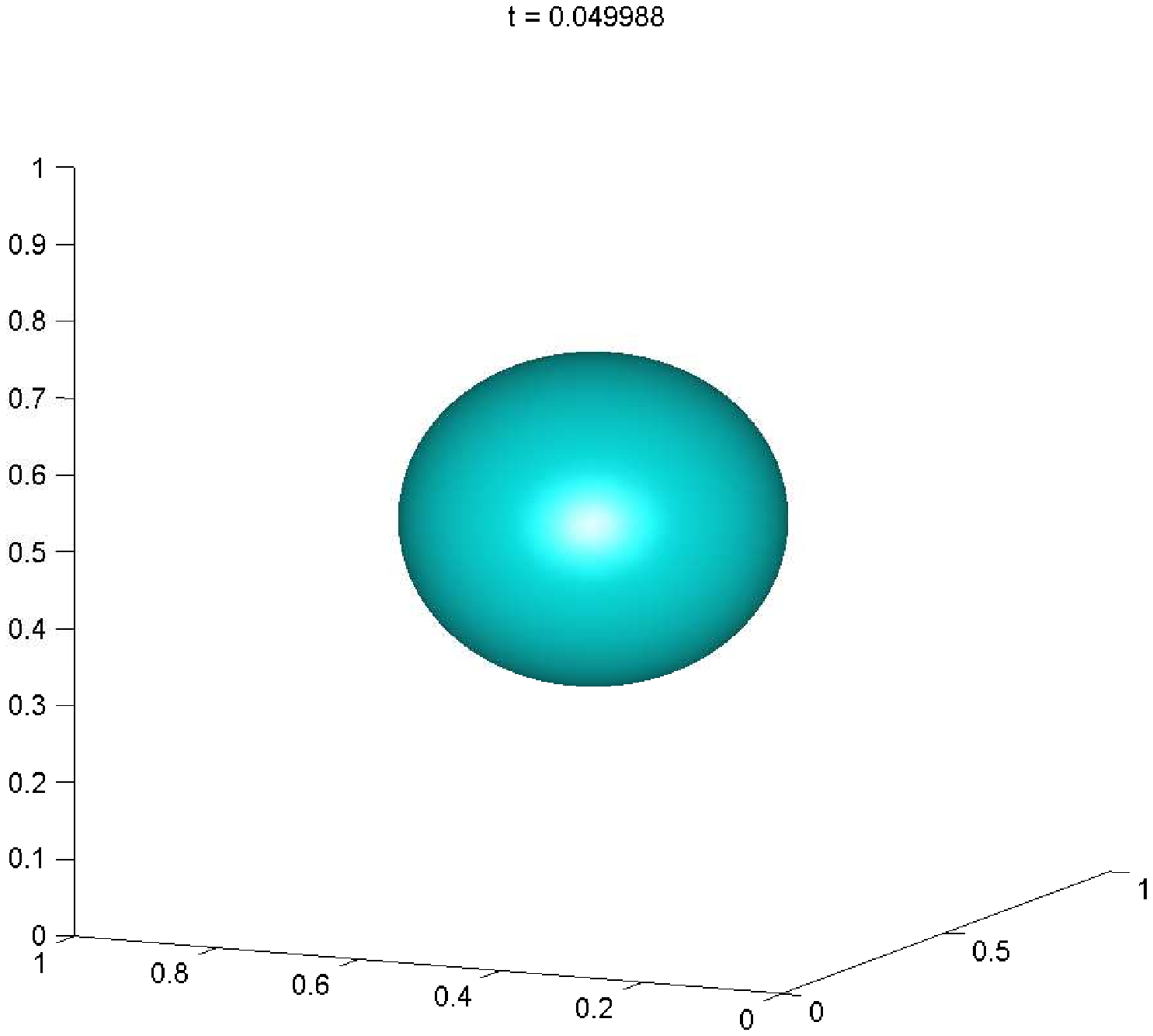}
\hspace*{\fill}
\\
\hspace*{\fill}
\includegraphics[width=0.2\linewidth]{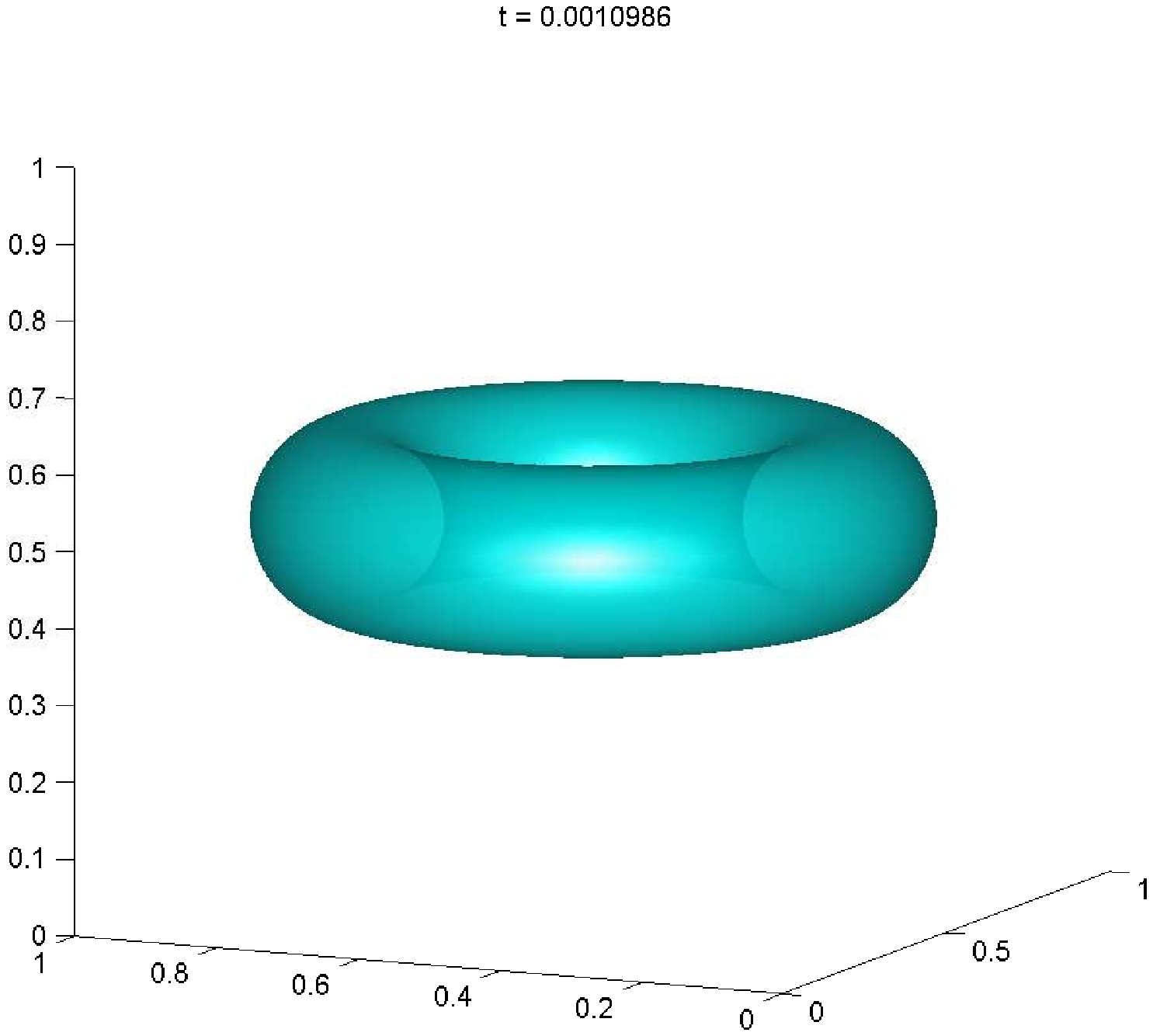}\hfill
\includegraphics[width=0.2\linewidth]{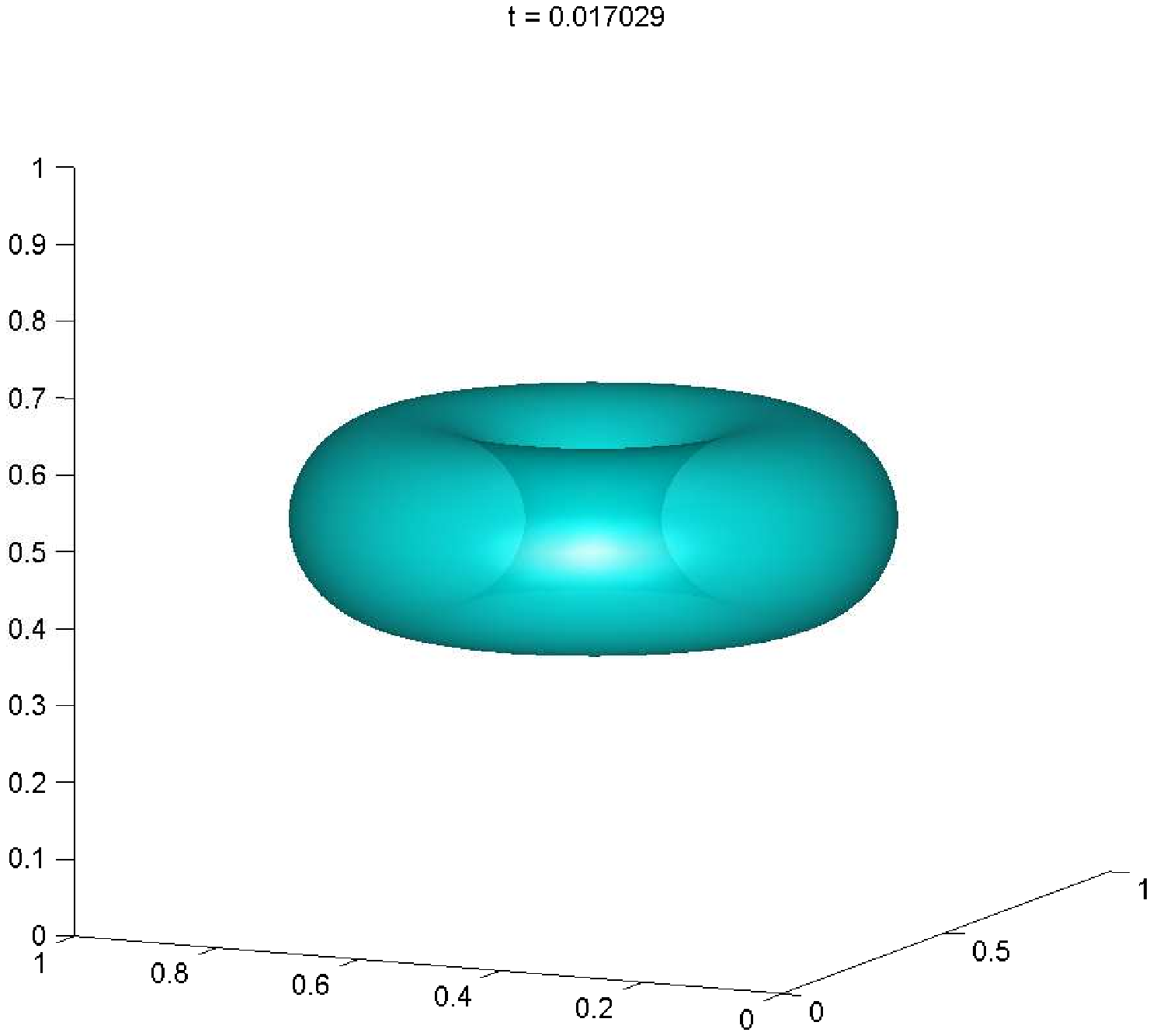}\hfill
\includegraphics[width=0.2\linewidth]{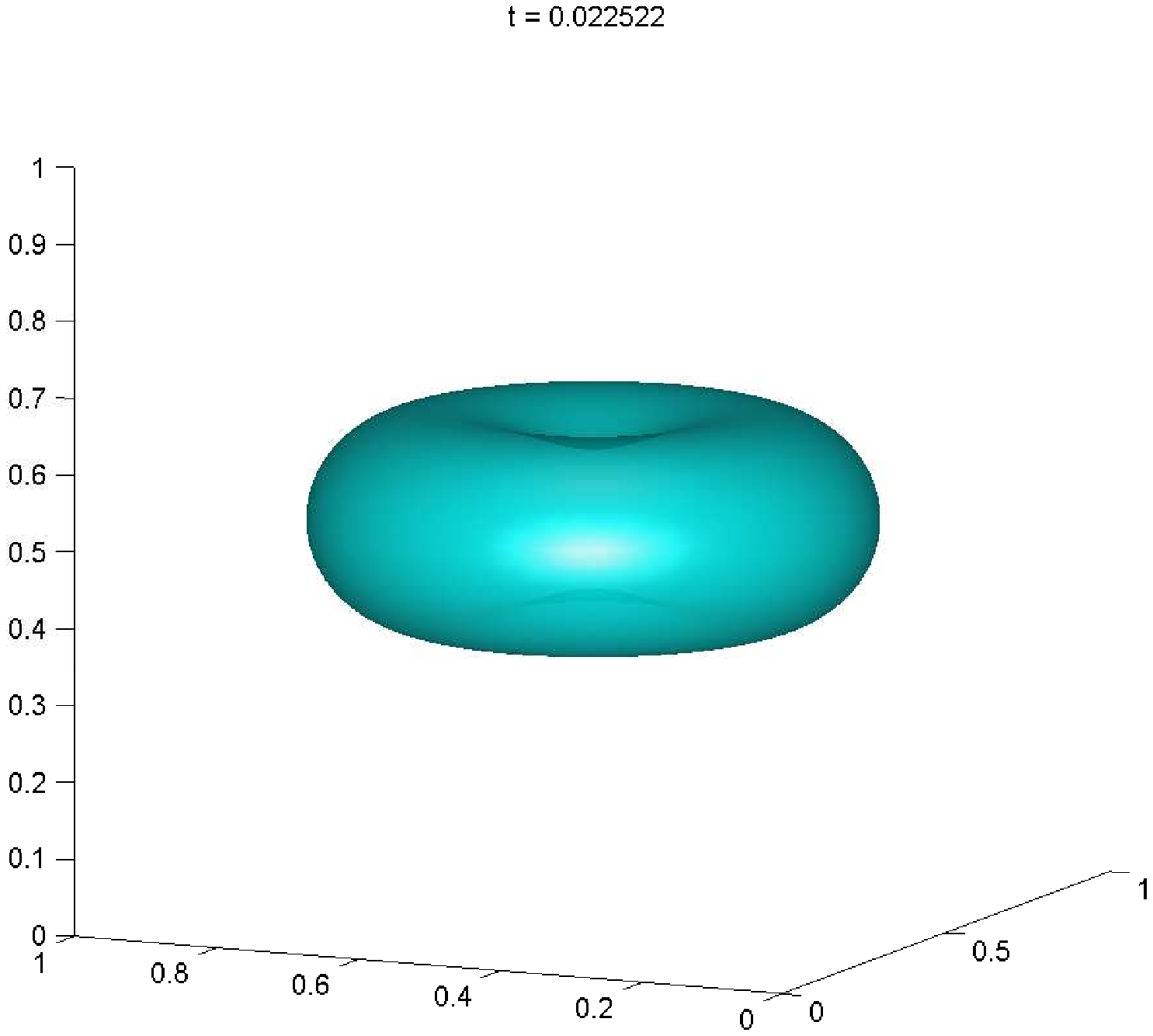}\hfill
\includegraphics[width=0.2\linewidth]{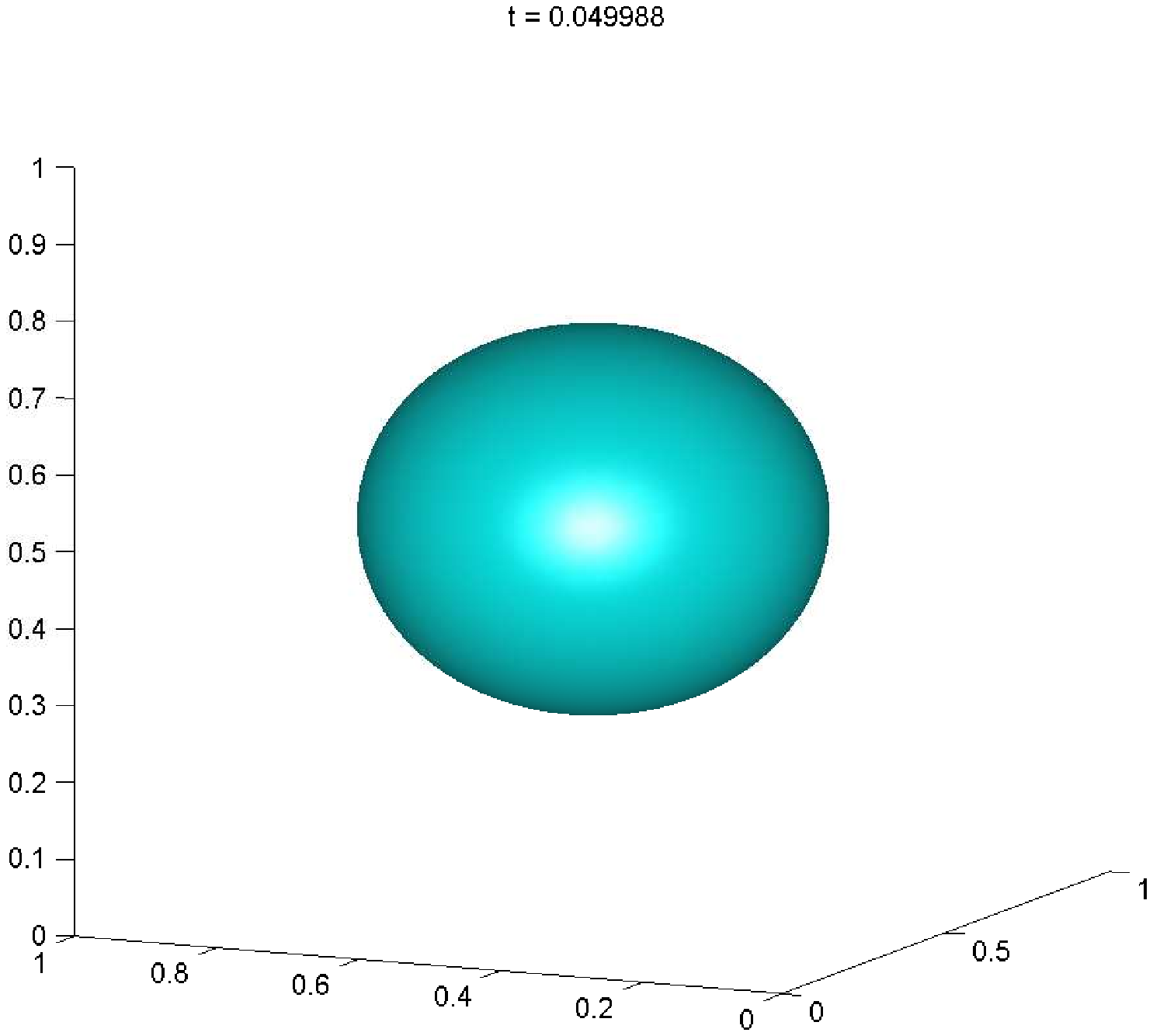}
\hspace*{\fill}
\caption{Evolution of a torus by mean curvature flow with conservation
of the volume. First line: classical model \eqref{eq:classic_model}.
Second line: modified model \eqref{eq:new_model}.}
\label{fig:torus}
\end{figure}

\begin{figure}[htbp]
\centering
\includegraphics[width=0.45\linewidth]{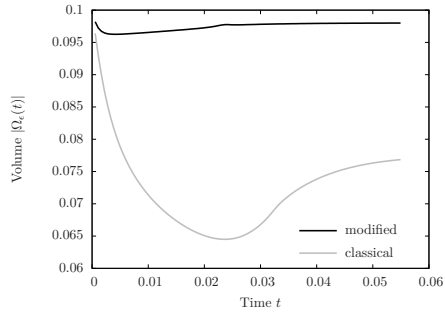}
\caption{Torus example: volume against time for both models
\eqref{eq:classic_model} and \eqref{eq:new_model}.}
\label{fig:vol_torus}
\end{figure}

\section{Conclusion}

We introduced in this article a modified phase field model for the
approximation of mean curvature flow with a forcing term. We rigorously
proved its convergence with the same order as the classical Allen--Cahn
equation: $O(\eps^2 \abs{\log \eps}^2)$.

We formally derived this model to the case of conserved mean curvature
flow. We observed numerically an $O(\eps^2)$ error for the conservation
of the volume, whereas the classical conserved Allen--Cahn equation just
showed an $O(\eps)$ error in our simulations.

\paragraph{Acknowledgements.}
The authors would like to thank Eric Bonnetier and Val\'erie Perrier
for their advice and fruitful discussions.

\bibliographystyle{abbrv}
\bibliography{biblio}

\end{document}